\documentclass[11pt,a4paper,english,final]{amsart}

\usepackage{csquotes}
\MakeOuterQuote{"}

\usepackage{amsmath} 				
\usepackage{amsfonts} 				
\usepackage{mathtools}				
\usepackage{calc}
\usepackage{graphicx} 				
\usepackage[top=2.70cm, bottom=2.8cm, left=2.8cm, right=2.8cm]{geometry} 		
\usepackage{lettrine} 				
\usepackage{amssymb}				
\usepackage{hyperref}				
\usepackage{array}					
\usepackage{textcomp}				
 
\usepackage{setspace}			
\usepackage[svgnames,table]{xcolor} %
\usepackage{lipsum}				%
\usepackage{amsthm}                       		%
\usepackage{tikz} 					%
\usetikzlibrary{shapes,snakes}
\tikzstyle{mybox} = [fill=gray!20,    rectangle, inner sep=10pt, inner ysep=10pt]
\usepackage{verbatim}
\usepackage{mathrsfs} 				%
\usepackage{caption}
\usepackage[toc,page]{appendix} %
\usepackage{esint}
\usepackage{enumitem}					%
\usepackage{slashed}					%
\usepackage[numbers,sort]{natbib}			%

\usepackage{ifdraft}
\ifdraft{%
\usepackage{fancyhdr}%
\pagestyle{fancy}%
\fancyhead[LO,LE]{\slshape \thepage}%
\fancyhead[RO,RE]{\slshape \today}%
\fancyfoot[]{}
}{}
\usepackage[parfill]{parskip}
\usepackage{showkeys}				%

\usepackage[obeyDraft]{todonotes}
\usepackage[]{marginnote}

\newcounter{marnote}

\usepackage{tikz-cd}
\usetikzlibrary{babel}

\newtheorem{thm}{\normalfont\bfseries\MakeUppercase{Theorem}}[section]

\newtheorem*{thm*}{\normalfont\bfseries{Theorem}}
\newtheorem{lem}[thm]{\normalfont\bfseries{Lemma}} 
\newtheorem*{lem*}{\normalfont\bfseries{Lemma}}
\newtheorem{prop}[thm]{\normalfont\bfseries{Proposition}}

\newtheorem{rk}{\normalfont\bfseries Remark}[section]

\renewcommand{\leq}{\leqslant}
\renewcommand{\geq}{\geqslant}

\def \Hom {\mathbf{Hom}}

\def \SO {\mathbf{SO}}
\def \so {\mathfrak{so}}

\def \O {\mathbf{O}}

\def \L {\mathrm{L}}
\def \W {\mathrm{W}}

\def \R {\mathbf{R}}

\def \Z {\mathbf{Z}}

\def \N {\mathbf{N}}

\def \epsi {\varepsilon}
\def \loc {\text{\rm loc}}

\def \t {{}^t}

\def \eg {\textit{ e.g.}}
\def \ie {\textit{ i.e.}}
\DeclareMathOperator{\laplace}{\bigtriangleup}
\DeclareMathOperator{\supp}{supp}

\DeclareMathOperator{\Rm}{\mathbf{Rm}}
\def \II {\mathbf{II}}

\let\div\relax
\DeclareMathOperator{\div}{div}
\DeclareMathOperator{\grad}{grad}
\DeclareMathOperator{\curl}{curl}

\DeclareMathOperator{\Vol}{Vol}

\DeclareMathOperator{\tr}{Tr}

\def \weakstar {\stackrel{\ast}{\rightharpoonup}}

\def \eg {\textit{e.g.} }
\newcommand{\euc}{{\mathfrak{e}}}

\def\XXint#1#2#3{{\setbox0=\hbox{$#1{#2#3}{\int}$ }
\vcenter{\hbox{$#2#3$ }}\kern-.6\wd0}}

\newcommand{\dd}{{\rm d}}
\newcommand{\DD}{{\rm D}}
\makeatletter
\@namedef{subjclassname@2020}{%
  \textup{2020} Mathematics Subject Classification}
\makeatother

\numberwithin{equation}{section}
\numberwithin{figure}{section}
\title[Weak continuity of curvature for connections in $\L^p$]{Weak continuity of curvature for connections in $\L^p$}
\author{Gui-Qiang G. Chen and Tristan P. Giron}
\begin{document}
\thanks{
The research of Gui-Qiang G. Chen was supported in part by
the UK
Engineering and Physical Sciences Research Council Award
EP/L015811/1 and EP/V008854,
and the Royal Society--Wolfson Research Merit Award WM090014 (UK).
The research of Tristan P. Giron  was supported in part by the UK
Engineering and Physical Sciences Research Council Award
EP/L015811/1.
}
\email{Gui-Qiang G. Chen: chengq@maths.ox.ac.uk} 
\email{Tristan P. Giron: tristan.giron@gmail.com} 
\address{Mathematical Institute, University of Oxford, Andrew Wiles Building, Woodstock Road, Oxford OX2 6GG, United Kingdom}

\maketitle

\begin{abstract}
We study the weak continuity of two interrelated non-linear partial differential equations, the Yang--Mills equations and the Gau\ss--Codazzi--Ricci equations, involving $\L^p$-integrable connections. Our key finding is that underlying cancellations in the curvature form, especially the div-curl structure inherent in both equations, are sufficient to pass to the limit in the non-linear terms. We first establish the weak continuity of Yang--Mills equations and prove that any weakly converging sequence of weak Yang--Mills connections in $\L^p$ converges to a weak Yang--Mills connection. We then prove that, for a sequence of isometric immersions with uniformly bounded second fundamental forms in $\L^p$, the curvatures are weakly continuous, which leads to the weak continuity of the Gau\ss--Codazzi--Ricci equations with respect to sequences of isometric immersions with uniformly bounded second fundamental forms in $\L^p$. Our methods are independent of dimensions and do not rely on gauge changes.
\end{abstract}

\maketitle
\setcounter{tocdepth}{1}
\tableofcontents

\section{Introduction}
We are concerned with sequences of connections on principal bundles, which are central objects of study in several important problems
in analysis, geometry, and mechanics, such as Yang--Mills equations, harmonic maps, wave maps, and isometric immersions
({\it cf.} \eg
\cite{csw-isometric,freire-muller-struwe,freire-muller-struwe-2,nakajima-compactness-yang-mills,riviere-park,riviere-variations,uhlenbeck-connections,sedlacek,donaldson-kronheimer,chen-li,helein} and the references therein) from the viewpoint of conservation laws (see \eg \cite{evans-weak,dafermos}).
When the curvatures of such sequences are controlled in sufficiently strong norms (to be precise, in scaling-invariant norms), it is possible to construct suitable gauges so as to extract a gauge-equivalent weakly $\W^{1,p}$--converging subsequence of connections, for example, by means of the Uhlenbeck gauge-fixing theorem \cite{uhlenbeck-connections}.
However, in many cases, the curvatures cannot be controlled directly in such norms, so that Uhlenbeck's gauge construction is no longer applicable, or only up to a singular
set ({\it cf}. \cite{nakajima-compactness-yang-mills}).
One of the objectives of this paper is to investigate several non-linear problems involving sequences of connections with a uniform $\L^p$ bound.

We study sequences of Yang--Mills connections as well as isometric immersions of Riemannian manifolds, under some mild boundedness assumptions. Our key finding is that the underlying cancellations in the curvature forms, especially
the div-curl structure inherent to each of these problems,
are sufficient to pass to limits in the non-linear terms.
We establish the weak continuity of both the Yang--Mills equations and the Gau\ss--Codazzi--Ricci
equations with respect to the weak $\L^p$ convergence. Our methods may be useful in other problems involving connections and curvatures.

Our approach is inspired from, and motivated by, the theory of compensated compactness, first initiated
by Murat \cite{murat-compcomp1} and Tartar \cite{tartar-pde} in Euclidean spaces.
Our results are similar in spirit to other weak continuity results,
such as those for harmonic maps \cite{helein} or wave maps \cite{freire-muller-struwe,freire-muller-struwe-2};
see also \cite{evans-weak} and the references therein.
As mentioned earlier, both our methods and results are independent of the gauge invariance of the problems under consideration, different from previous results such as 
\cite{uhlenbeck-connections,tian-gauge,riviere-variations,nakajima-compactness-yang-mills}
and the references cited therein for the Yang--Mills equations,
and the results in the forthcoming paper \cite{chen-giron-gauge} for the Gau\ss--Codazzi--Ricci equations.
For the Gau\ss--Codazzi--Ricci equations, the results of the present paper generalize
those in \cite{csw-weak-cty,chen-li},
where additional boundedness assumptions are required.

\subsection{Weak continuity of the Yang--Mills equations}
Our first result is the weak continuity of the Yang--Mills equations with respect to sequences of Yang--Mills connections.

Let $G$ be a matrix group, and let $P$ be a principal $G$-bundle over a Riemannian manifold $M$. The primary objects of interest for the Yang--Mills theory are connections $\nabla_A$ on $P$
and their curvatures, denoted by $\Omega_A$,
both of which are represented by differential forms taking values in the Lie algebra $\mathfrak{g}$ of the Lie group $G$ (see §\ref{subsection-principal} for our notations). 
The Yang--Mills action is the functional
\begin{equation}\label{eqn-action-ym}
\mathcal{YM}:\nabla_A\mapsto\int_M |\Omega_A|^2.
\end{equation}
Here and hereafter, we always drop ${\rm d}V$ or $\dd x$ in the integrals over manifolds or domains for simplicity.
The Yang--Mills action is finite for connections in $\W^{1,2}\cap\L^4$.
A connection is said to be Yang--Mills if it is a critical point of the Yang--Mills action,
{\it i.e.}, it satisfies
\begin{equation}\label{eqn-yang-mills}
\DD_A^*\Omega_A=0,
\end{equation}
where $\DD_A^*$ is the adjoint of the exterior derivative $\DD_A$ associated with the connection $\nabla_A$;
see \S \ref{section-connections} for further details and notations used throughout the paper.

We recall that due the gauge invariance, the Yang-Mills equations \eqref{eqn-yang-mills} are not elliptic with respect to the connection $\nabla_A$ (see \eg \cite{donaldson-kronheimer}). Previous results, most notably Uhlenbeck's compactness theorem \cite{uhlenbeck-connections}, established weak compactness for sequences of $\W^{1,p}$ connections with curvatures uniformly bounded in $\L^{n/2}$ by selecting adequate gauges. 
Our first result (Theorem \ref{thm-yang-mills} below) proves the stability of the class of Yang--Mills connections with respect to the weak convergence in $\L^p$ without gauge changes:
{\it Let $\nabla_{A_\epsi}$ be a sequence of weak Yang--Mills connections, converging weakly in $\L^4_{\rm loc}$
	to a connection $\nabla_A$.
	Assume that the sequence has locally uniformly bounded energy, \ie, on any bounded domain $E\subset M$,
	\begin{equation}\label{eqn-hypothesis-yang-mills}
	\sup_{\epsi>0} \Vert \Omega_{A_\epsi} \Vert_{\L^2(E)}<\infty.
	\end{equation}
	Then the weak limit connection $\nabla_A$ is Yang--Mills with corresponding curvature $\Omega_{A}$ and locally finite energy.
}

Hypothesis \eqref{eqn-hypothesis-yang-mills} states that the connection sequence $\nabla_{A_\epsi}$
has a uniformly finite Yang--Mills energy locally.

In \cite{nakajima-compactness-yang-mills,tian-gauge}, it is proved that,
under the hypothesis of finite energy \eqref{eqn-hypothesis-yang-mills}, up to a subsequence,
$\nabla_{A_\epsi}$ gauge-converges to a weak Yang--Mills connection, away from a singular set of codimension $4$.
Our result contrasts starkly with this situation: weak $\L^4$ convergence is sufficient to ensure no concentration of energy occurs
in {\it any} dimension, in any fixed trivialization of the bundle $P$.
We remark that Theorem \ref{thm-yang-mills} can also be extended to approximate solutions of the Yang--Mills equations, {\it i.e.},
sequences of connections $\nabla_{A_\epsi}$ such that $\DD^*_{A_\epsi}\Omega_{A_\epsi}\weakstar0$ in the sense of
distributions as $\epsi\to 0$.

\subsection{Weak continuity of the Gau\ss-Codazzi-Ricci equations}

The next result is the weak continuity of the Gau\ss-Codazzi-Ricci equations with respect to sequences of isometric immersions with uniformly bounded second fundamental forms in $\L^p_{\loc}$ for $p>2$, which improves the earlier results of \cite{chen-li,csw-weak-cty}.
The main point is to prove that curvatures are weakly continuous with respect to sequences of isometric immersions with uniformly bounded second fundamental form in $\L^p_{\loc}$ for $p>2$.

The motivations for such results are to understand the relationship between intrinsic and extrinsic properties of manifolds. Since the work of Nash \cite{nash-1,nash-smooth} and Kuiper \cite{kuiper}, it is known that uniform limits of isometric immersions and embeddings may fail to be isometric. However, it is easy to show (see Appendix \ref{section-very-weak}) that the limit of a sequence of isometric immersions with bounded mean curvature is still isometric. Curvatures are central objects of interest for Riemannian metrics, and it is known that the relationship between intrinsic and extrinsic curvature plays an important role in rigidity questions: see \eg \cite{cdls, gromov-columbus,pogorelov} and the references therein. We therefore enquire whether the curvature equations hold in the limit for a suitably converging sequence of isometric immersions.

Another motivation is the construction of isometric immersions into prescribed target spaces, which is another important question on the topic of isometric immersions. In \cite{csw-isometric} (and the subsequent works \cite{chw-1,chw-2}), the authors were able to construct isometric immersions by solving the Gau\ss--Codazzi--Ricci system via weak compactness methods.  Such results are also related with intrinsic elasticity: see \cite{ciarlet-2,ciarlet-3,ciarlet-larsonneur} and the references therein.

Let $(M,g)$ be a
Riemannian manifold.
Recall that an immersion is a map $u\,:\,(M,g)\rightarrow (\widetilde{M},\tilde{g})$,
where $(\widetilde{M},\tilde{g})$ is another
Riemannian manifold,
such that the tangent map $\dd u\,:\,TM\rightarrow T\widetilde{M}$ is injective.
It is isometric if the metric $\tilde{g}$ pulls back to $g$ under $u$, or equivalently $u^*\tilde{g}=g$.
In local coordinates on $M$, this reads
\begin{equation}\label{eqn-isometry}
\tilde{g}(u_*\frac{\partial}{\partial x^i},u_*\frac{\partial}{\partial x^j}) = g_{ij}.
\end{equation}
A weak $\W^{2,p}_{\loc}$ isometric immersion is a $\W^{2,p}_{\loc}$ map that satisfies \eqref{eqn-isometry}
and is additionally a bi-Lipschitz homeomorphism onto its image
(see \S \ref{section-defi-immersion}).

The immersion map identifies the tangent bundle $TM$ as a sub-bundle of $T\widetilde{M}$.
The orthogonal complement of $TM$ in $T\widetilde{M}$ is the normal bundle $NM$,
which may also be viewed as a sub-bundle of $T\widetilde{M}$.
Moreover, we recall that the Levi-Civita connection on $\widetilde{M}$ induces a connection on the normal bundle and a second fundamental form $\II$
which is an element of $\Hom(TM\times TM,NM)$ and, in components, satisfies $\II_{ij}^k = \II_{ji}^k$ (see \eg \cite{spivak} for a comprehensive presentation).
The second fundamental form and the normal connection characterize the extrinsic geometry of the immersion;
together, they satisfy the Gau\ss--Codazzi--Ricci equations, linking the extrinsic geometry and the intrinsic geometry of the manifold.
Our second result, Theorem \ref{thm-main-isometric-immersions}, proves the following:
{\it Let $u_\epsi\,:\,(M,g)\rightarrow(\widetilde{M},\tilde{g})$ be a sequence of  isometric immersions, whose second fundamental forms $\II_\epsi$ are locally uniformly
	bounded in $\L^p_{\loc}$ for $p>2$. Then, up to a subsequence, $u_\epsi$ converges weakly to a weak $\W^{2,p}_{\loc}$ isometric immersion obeying
	the Gau\ss--Codazzi--Ricci equations.
	In particular, the extrinsic curvatures are weakly continuous with respect to sequence $u_\epsi$.}

The uniform boundedness of the second fundamental forms $\II_\epsi$ in $\L^p_{\loc}$ is equivalent
to the uniform boundedness of the Hessians of the isometric immersions $u_\epsi$ in $\L^p_{\loc}$.
Thus, the weak convergence (up to a subsequence) of the sequence of $u_\epsi$ follows directly
from the Banach--Alaoglu theorem.
The main point of the result described above ({\it cf.} Theorem \ref{thm-main-isometric-immersions})
is to verify that the limit point of such a subsequence
still satisfies the Gau\ss--Codazzi--Ricci equations. Indeed, recall that a priori, the extrinsic geometry depends on the second order derivatives of $u$, while the intrinsic geometry depends on second order derivatives of $g$, and thus third order derivatives of $u$; moreover, the relation between them is given by the Gau\ss--Codazzi--Ricci system, a non-linear, mixed-type system of PDE.

The case $p=2$ requires special attention. In general, it is not true that a sequence of immersions
converging weakly in $\W^{2,2}_{\loc}$ converges to another immersion determined by the corresponding
Gau\ss--Codazzi--Ricci equations, even for surfaces in $\R^3$ (see \cite{langer} for a counter-example).
However, we show that this is the case for a sequence of codimension-one isometric
immersions (see Theorem \ref{thm-weak-cty-equi}).

Our results improve the weak continuity and pre-compactness results obtained in \cite{chen-li} (that require $p>n$)
to the more general case of $p>2$ and arbitrary Riemannian backgrounds by using an ad hoc construction, rather than relying on the realisation theorem for submanifolds (see \cite{chen-li,ciarlet-larsonneur,mardare-lp}).
In addition, we also study the borderline case $p=2$.
We remark that, in this paper, we focus our analysis on the sequences of isometric immersions.
Our approach can be applied to the sequences of approximate solutions
of the Gau\ss--Codazzi--Ricci equations,
but this presents additional difficulties; see \cite{chen-giron-gauge} for further discussions.

\medskip
This paper is organized in five sections.
In \S \ref{section-connections}, we first recall the basics of the Yang--Mills equations
and the Gau\ss--Codazzi--Ricci equations on Riemannian manifolds, and then give the precise description of the main theorems,
Theorems \ref{thm-yang-mills} and \ref{thm-main-isometric-immersions}.
Theorem \ref{thm-yang-mills} is proved in \S \ref{section-yang-mills}. Next, in \S \ref{section-immersions}, we provide the proof of Theorem \ref{thm-main-isometric-immersions}
for the weak continuity of the Gau\ss--Codazzi--Ricci equations and related problems.
In \S \ref{section-generalisation}, we present extensions of Theorem \ref{thm-main-isometric-immersions}, including the case $p=2$, the case of semi-Riemannian manifolds, and the problem of minimising the $\L^p$ norm of the second fundamental form for isometric immersions. There are two appendices. Appendix \ref{section-very-weak} presents a general result about the stability of metrics along sequences of isometric immersions. In Appendix \ref{appendix-div-curl}, we review some basic material and notations
used in the proofs concerning div-curl lemmas.

\section{Setting and Main Theorems}\label{section-connections}

In this section, we give the precise setting of the main theorems, Theorems \ref{thm-yang-mills}--\ref{thm-main-isometric-immersions}. Both rely on a common observation concerning the properties of the curvature form, presented in detail in \S\ref{subsection-method}. 
First, in \S \ref{subsection-principal}, we recall some basics on principal bundles. This material is well-known; standard references are \eg \cite{uhlenbeck-wehrheim}.
In \S\ref{subsection-yang-mills}, we present a precise statement of our first main theorem, Theorem \ref{thm-yang-mills},
about the Yang--Mills equations. 
Our second main theorem, Theorem \ref{thm-main-isometric-immersions}, along with the additional necessary pre-requisites
about the Gau\ss--Codazzi--Ricci equations, is introduced in \S\ref{subsection-immersions-sequences}--\S\ref{section-gcr}.
Finally, in \S\ref{subsection-method}, we sketch the proofs for
both Theorems \ref{thm-yang-mills}--\ref{thm-main-isometric-immersions}.

\subsection{Principal bundles}\label{subsection-principal}
Throughout this paper, we always assume that $M$ is an $n$-dimensional
Riemannian manifold with required regularity.
Let $P$ be a principal
bundle over $M$,
and let the structure group $G$ of $P$ be a matrix group.
If $P$ is equipped with a scalar product on the fibers,
we assume that $G$ is a subgroup of the special orthogonal group preserving the prescribed inner product on the fiber.
We denote by $\mathbf{Ad}(P)$ the adjoint bundle, whose fiber is the Lie algebra $\mathfrak{g}$ generated by $G$.
Finally, $\Gamma(M,P)$ denotes the set of sections of bundle $P$ over $M$.

Let $A,B\in\Lambda^k(M,\mathbf{Ad}(P)):=\Gamma(M,\mathbf{Ad}(P)\otimes \wedge^kT^*M)$ be $\mathbf{Ad}(P)$-valued $k$-forms.
We denote by $[\cdot,\cdot]$ the Lie bracket on the Lie algebra $\mathfrak{g}$, which induces a Lie bracket defined fiberwise
on $\mathbf{Ad}(P)$, and by $[A\wedge B]$ the wedge product of two $k$-forms, the values of which are combined through the Lie bracket.
As $\mathfrak{g}$ is a Lie algebra and $G$ is a compact Lie group,
there exists a unique $G$-equivariant scalar product on the fibers of $\mathbf{Ad}(P)$,
denoted by $\langle\cdot,\cdot\rangle$.
The scalar product is compatible with the Lie bracket in the sense that the triple product identity
$\langle [\cdot,\cdot],\cdot\rangle=\langle \cdot,[\cdot,\cdot]\rangle$ holds.
In addition, we rescale it in such a way that $|[\xi,\zeta]|\leq |\xi||\zeta|$ for arbitrary $\xi,\zeta\in\mathfrak{g}$.
By abuse of notation, $\langle\cdot,\cdot\rangle$ also denotes the scalar product on $\Lambda^k(M,\mathbf{Ad}(P))$ obtained
by combining the scalar product on $\mathfrak{g}$ and the Riemannian metric $g$ on $M$.

The wedge product of $\mathfrak{g}$-valued forms, $[\alpha\wedge\beta]\in\Lambda^{p+q}(M,\mathfrak{g}),	$
is defined as follows:
for $\alpha\in\Lambda^p(M,\mathfrak{g})$ and $\beta\in\Lambda^q(M,\mathfrak{g})$,
\begin{equation*}
[\alpha\wedge\beta](X_1,\dots,X_{p+q}) =\sum_{\sigma\in \text{Sh}_{p+q}}\epsi(\sigma)[\alpha(X_{\sigma(1)},\dots,X_{\sigma(p)}),\beta(X_{\sigma(p+1)},\dots,X_{\sigma(p+q})],
\end{equation*}
where $\epsi(\sigma)$ is the signature of shuffle $\sigma\in \text{Sh}_{p+q}$,
{\it i.e.}, the monotone bijections of $\Z_{p+q}$. Then we have
\begin{align*}
&[\alpha\wedge\beta] = (-1)^{pq+1}[\beta\wedge\alpha],\\
&[[\alpha,\alpha],\alpha]=0,\\
&\dd[\alpha\wedge\beta]=[\dd\alpha\wedge \beta]+(-1)^{pq}[\alpha\wedge \dd\beta].
\end{align*}

We denote
$\mathcal{A}(P)$ the space of
connections on $P$; 
$\mathcal{A}(P)$ is an affine space. Let us fix a base connection $\widetilde{\nabla}$.
Then any connection in $\mathcal{A}(P)$ differs from $\widetilde{\nabla}$ by an $\mathbf{Ad}(P)$-valued one-form.
Up to a choice of local trivialization of $P$, we can view any connection $\nabla_A\in\mathcal{A}(P)$
as being given by a $\mathbf{Ad}(P)$-valued differential one-form $A$, so that
\begin{equation*}
\nabla_A=\widetilde{\nabla}+A \qquad \mbox{with $\,A\in \Gamma(M, \mathbf{Ad}(P)\otimes T^*M)$}.
\end{equation*}
Equivalently, we may consider a covering of $P$ by trivializing sets $U^j$, so that
a connection $\nabla_A$ is then given as a list of $\mathfrak{g}$-valued one-forms $A^j$.

The main invariant of a connection is its curvature.
For a connection ${\nabla_A\in\mathcal{A}(P)}$, its curvature is
the two-form $\Omega_A\in \Lambda^2(M,\mathbf{Ad}(P))$
representing the action of $\nabla_A\circ\nabla_A$.
In a local trivialization $U^j$, we may write $\nabla_A=\dd + A^j$ so that its connection is
\begin{equation}\label{2.5a}
\Omega_A|_{U^j} = \dd A^j + \frac{1}{2}[A^j\wedge A^j].
\end{equation}

We are concerned with the weak continuity of $\Omega_A$ with respect to the sequences of connections
in some Sobolev spaces that are defined as follows:
let $k\in\N$ and $1\leq p\leq\infty$. Denote $\mathcal{A}^{k,p}(P)$ as the spaces of connections on $P$ defined by
\begin{align*}
\mathcal{A}^{k,p}_{\loc}(P) %
= \big\{ \widetilde{\nabla}+A\, :\, A\in\W^{k,p}_{\loc}(M,\mathbf{Ad}(P)\otimes T^*M)\big\}.
\end{align*}
Recall that $\mathbf{Ad}(P)\otimes T^*M$ is a vector bundle equipped with a scalar product.
Thus, for all $k\in\N$ and $1\leq p \leq \infty$, the Sobolev space $\W^{k,p}_{\loc}$ of sections of bundle $\mathbf{Ad}(P)\otimes T^*M$
is defined as the completion of smooth
sections with respect to the natural Sobolev norms locally. We note that the curvature is well-defined in the sense of distributions for connections in $\mathcal{A}^{0,2}_\loc(P)$.

\subsection{Yang--Mills functional}\label{subsection-yang-mills}
Recall (see \eg \cite{donaldson-kronheimer}) that a connection $\nabla_A\in\mathcal{A}(P)$ is a Yang--Mills connection if it is a critical point of the Yang--Mills action integral $\mathcal{YM}(\nabla_A)$, defined by
\begin{equation}\label{eqn-ym-functional}
\mathcal{YM}(\nabla_A):=\Vert\Omega_A\Vert_{\L^2}^2=\int_M |\Omega_A|^2.
\end{equation}
The Euler--Lagrange equation of this functional is the Yang--Mills equation,
\begin{equation}\label{eqn-yang-mills-b}
\DD^*_A\Omega_A=0,
\end{equation}
where $\DD_A^*$ refers to the adjoint of $\DD_A$, the covariant derivative associated with $\nabla_A$.
Recall that, for any $k$-form $\omega\in\Gamma(M,\mathfrak{g}\otimes\wedge^kT^*M)$,
its covariant derivative is given by
\begin{equation*}
\DD_A\omega = \dd\omega + [A\wedge\omega].
\end{equation*}
Thus, the adjoint covariant derivative is given by
\begin{equation*}
\DD_A^*\omega=\delta \omega -(-1)^{(n-k)(k-1)} *[A \wedge *\omega]
\end{equation*}
for any arbitrary $k$-form $\omega \in\Gamma(M,\mathfrak{g}\otimes \wedge^k T^*M)$,
where $\delta$ is the co-differential operator on $k$-forms, and $*$ is the Hodge duality operator for $n=\dim(M)$.

On a compact Riemannian manifold without boundary, the Yang--Mills functional \eqref{eqn-ym-functional} is finite for connections ${\nabla_A\in\mathcal{A}^{1,2}(P)\cap\mathcal{A}^{0,4}(P)}$.
By Sobolev embeddings, $W^{1,p}$ connections for $p\geq \max(\frac{4n}{n+4},2)$ are in this class.
Indeed, such connections are in $\L^4$ so that $[A\wedge A]\in\L^2$ and $\Omega_A\in\L^2$.
In such a space, the critical points of \eqref{eqn-ym-functional} satisfy equation \eqref{eqn-yang-mills-b} in the weak sense;
that is, for all $\phi\in\Lambda^1(M,\mathbf{Ad}(P))$,
\begin{equation}\label{eqn-yang-mills-weak}
\int_M \langle \Omega_A,\DD_A\phi\rangle=0.
\end{equation}
This is the weak form (or distributional form) of equation \eqref{eqn-yang-mills-b}.

We now state our first main theorem.

\begin{thm}\label{prop-yang-mills}\label{thm-yang-mills}
	Let $P$ be a principal $G$-bundle over $M$, and let $\widetilde{\nabla}$ be a fixed connection.
	Consider a sequence of weak Yang-Mills connections $\nabla_{A_\epsi}\in\mathcal{A}^{0,4}(P)$,
	which can be written as $\nabla_{A_\epsi}=\widetilde{\nabla}+A_\epsi$
	for the $\mathbf{Ad}(P)$-valued one-form $A_\epsi$, so that $\DD^*_{A_\epsi}\Omega_{A_\epsi}=0$.
	Assume that
	\begin{subequations}
	\begin{align}
	&A_\epsi\rightharpoonup A \text{ weakly in } \L^4_{\loc}(M,\mathbf{Ad}(P)\otimes T^*M),\label{eqn-finite-action-a}\\
	&\sup_{\epsi>0} \Vert \Omega_{A_\epsi} \Vert_{\L^2(E)}<\infty \qquad\mbox{for any bounded domain $E\subset M$}.
	\label{eqn-finite-action}
	\end{align}
	\end{subequations}
	Then $\widetilde{\nabla}+A$ is a weak Yang--Mills connection.
\end{thm}

The first hypothesis \eqref{eqn-finite-action-a} is the weak convergence hypothesis.
The second hypothesis \eqref{eqn-finite-action} states that the sequence considered has local finite energy.
We recall that, as a distribution, the curvature form makes sense for $\L^2_\loc$ connections.
Hypothesis \eqref{eqn-finite-action} states that this distribution is in fact an $\L^2_\loc$ function, for each $\epsi>0$. This does not imply that the connection itself is bounded in $\W^{1,2}_\loc$ in general.

We remark that the same result in Theorem \ref{prop-yang-mills} holds if the sequence of Yang-Mills connections $\nabla_{A_\epsi}$
is relaxed as a sequence of approximate Yang--Mills connections that satisfy
$$
\DD^*_{A_\epsi}\Omega_{A_\epsi}\weakstar0 \qquad \text{ in $\mathcal{D}'$},
$$
beside conditions \eqref{eqn-finite-action-a}--\eqref{eqn-finite-action};
that is, the weak limit is still a weak Yang--Mills connection.

\subsection{Weak immersions of Riemannian manifolds}\label{subsection-immersions-sequences}

We first define a notion of \emph{weak immersion}.
For $p\geq1$, we say that a map $u\,:\,M\rightarrow\R^N$ is a $\W^{2,p}_{\loc}$ immersion if it is both a $\W^{2,p}_{\loc}(M,\R^N)$ map
and a bi-Lipschitz homeomorphism onto its image.

It is clear that, under the above hypothesis,
$u(M)$ is a Lipschitz submanifold of $\R^N$.
If $p\leq n$, $u$ is not $C^1$ in general, so that $u$ may fail to be a differentiable immersion in the classical sense. Moreover, when $u$ is an isometric immersion in the classical sense,
$u$ is necessarily a homeomorphism onto its image
and is also Lipschitz by the isometry condition.

The notion of weak immersions arises naturally in the study of Sobolev immersions whenever the immersions are not be assumed
to be $C^1$ {\it a priori}; see {\it e.g.} \cite{riviere-park,riviere-willmore}.
In the case that $M$ is a two-dimensional surface given as the graph of a $\W^{2,2}$ $\R$-valued function (so that we take $N=3$),
it is known from \cite{toro,muller-sverak} that the manifold induced by the graph admits a bi-Lipschitz parametrization.
The definition given above generalizes this fact; see {\it e.g.} \cite{toro} for the examples of weak immersions of surfaces in $\R^3$ that are not $C^1$.

A weak immersion of a Riemannian manifold $(M,g)$ is isometric if it satisfies
$\langle \dd u,\dd u\rangle=g$ pointwise almost everywhere; this is equivalently written as $u^*\mathfrak{e}=g$. We note that the induced metric $g_{ij}=\langle \partial_iu,\partial_ju\rangle$ is an $\L^\infty_\loc$ function on $M$. Similarly, it is easy to see that, for a local basis of the normal bundle $\{\nu_1,\dots,\nu_{N-n}\}$,
$$
\II^k_{ij}=\langle \partial_{ij}u,\nu_k\rangle,\qquad 1\leq k\leq N-n,\quad 1\leq i,j\leq n,
$$
which are $\L_\loc^p$ functions.

A sequence of immersions $u_\epsi\,:\,M\rightarrow\R^N$ is said to be weakly convergent
in $\W^{2,p}_{\loc}\cap W^{1,\infty}_{\loc}$
if there exists a $\W^{2,p}_{\loc}$ immersion $u$ (as above) such that
$$
u_\epsi\weakstar u \qquad\mbox{in the weak* topology of $\W^{1,\infty}_{\loc}$},
$$
and
$$
\nabla^2u_\epsi\rightharpoonup\nabla^2u \qquad \mbox{weakly in $\L^p_{\loc}$},
$$
which will be written as $u_\epsi\rightharpoonup u$ in $\W^{2,p}_{\loc}\cap \W^{1,\infty}_{\loc}$ from now on. It is straightforward to verify that if $u_\epsi$ converges weakly to an immersion $u$ in $\W^{2,p}$ for $p>2$, then the first fundamental forms of $u_\epsi$ converge in $\L^\infty_\loc$ to that of $u$ (see Appendix \ref{section-very-weak} for a further discussion on this question). We note that the notions defined above for immersions into $\R^N$ may be extended in the usual way to the case where $\R^N$ is replaced by an arbitrary $N$-dimensional Riemannian manifold $(\widetilde{M},\tilde{g})$, by isometrically embedding $(\widetilde{M},\tilde{g})$ into a larger Euclidean space $\R^{\widetilde{N}}$, by Nash's theorem \cite{nash-smooth}.

\subsection{Adapted coframes and the Darboux bundle}\label{section-darboux}\label{section-gcr}

Let $(M,g)$ and $(\widetilde{M},\tilde{g})$ be Riemannian manifolds of dimensions $n$ and $N$ respectively, with $N>n$, and let $u\,:\,M\rightarrow \widetilde{M}$ be a weak isometric immersion, \ie\ a local bi-Lipschitz homeomorphism onto its image.
Before stating our second main theorem, Theorem \ref{thm-main-isometric-immersions},
we recall some basic notions about immersions of Riemannian manifolds.
The point of view adopted here and throughout the paper is that of Cartan's moving frame method; see {\it e.g.} \cite{spivak}.

We recall that the normal bundle is defined as the quotient bundle
$NM=T\widetilde{M}|_{u(M)}/TM$.
By orthogonality, at every point $x\in M$, the fiber $\big(T\widetilde{M}|_{u(M)}\big)_x$ decomposes into an orthogonal sum
\begin{equation*}
\big(T\widetilde{M}|_{u(M)}\big)_x=\big(T\widetilde{M}|_{u(M)}\big)^\top_x\oplus \big(T\widetilde{M}|_{u(M)}\big)^\perp_x\simeq TM_x\oplus NM_x.
\end{equation*}
Additionally, we recall that the normal bundle is independent of the metric on $M$.

Let $\mathscr{F}(T\widetilde{M})$ be the orthonormal
frame bundle of the tangent bundle. We recall that $\mathscr{F}(T\widetilde{M})$ is a principal $\SO(N)$-bundle, where $N=\dim(\widetilde{M})$. Local sections of $\mathscr{F}(T\widetilde{M})$ are local orthonormal frames. Similarly, $\mathscr{F}(T^*\widetilde{M})$ is the coframe bundle. Viewing coframes as lists of covectors, a local coframe $\{\tilde\theta^1,\dots,\tilde\theta^N\}$ is adapted to an immersion $u: M\rightarrow \widetilde{M}$,
if it is orthonormal in $T^*\widetilde{M}$ and (up to reordering) $\{u^*\tilde\theta^1,  \dots, u^*\tilde\theta^n\}$ are an orthonormal coframe on $M$.
In matrix notation, a coframe $\tilde\theta\in\mathscr{F}(T^*\widetilde{M})$ is adapted
if $\theta := u^*\tilde\theta$ is an orthonormal coframe on $M$. Adapted coframes are referred to as Darboux coframes (see also \cite{bbg}).

Let $\tilde\theta$ be a Darboux coframe, and let $\tilde\omega$ be the  associated connection form. Writing as before $\theta=\theta^\top=u^*\tilde\theta$, which is an orthonormal coframe on $M$, let us denote by $\omega$ the connection form associated to $\theta$. We recall that this means that $\theta$ and $\omega$ satisfy the first Cartan equation $\dd\theta=\omega\wedge\theta$ and that if $\theta$ is sufficiently smooth ($\theta\in\W^{1,2}$ is sufficient, see \cite{phd-giron}), such a connection form always exists and is unique. Thus by uniqueness, $\omega = u^*(\tilde\omega)=\omega^\top$, where we write in a local trivialization:
\begin{equation}\label{eqn-darboux-block}
\tilde\omega = \begin{pmatrix}
\omega^\top & \t\omega^{\II}\\
-\omega^{\II} & \omega^\perp
\end{pmatrix}\in\so(N).
\end{equation}
The notation $\omega^\top$ refers to the tangential part of the connection, and $\omega^\perp$ to the normal connection. For $\alpha$ an $n\times m$ matrix, $\t \alpha$ is the transpose matrix, which is an $m\times n$ matrix; the notation $\t \omega^\II$ is the transpose of the second fundamental form part of the connection $\tilde{\omega}$.

Writing out the curvature equations for such coframes, we find that, by definition, $\Omega = \dd\omega + \omega\wedge\omega= \Rm(\theta\wedge\theta)$, where the Riemann curvature tensor is viewed as an endomorphism on two-forms. In components, this yields
\begin{equation*}
	\Omega^i_j=\frac{1}{2}R^i_{jkm}\theta^k\wedge\theta^m.
\end{equation*}
Similarly, $\widetilde{\Omega}=\dd\tilde\omega+\tilde\omega\wedge\tilde\omega=\widetilde{\Rm}(\tilde\theta\wedge\tilde\theta)$.
It is well-known \cite{spivak} that from these two equations one deduces the Gauß--Codazzi--Ricci equations, which is the following system:
\begin{align}
\begin{cases}\label{system-gcr}
\t\omega^\II\wedge\omega^\II = \tau^*\Rm(u^*\theta\wedge u^*\theta)-\widetilde{\Omega}|_{TM\times TM},\\[1mm]
\dd\omega^{\II} = \omega^\top\wedge\omega^{\II} - \t\omega^{\II}\wedge\omega^\top+\widetilde{\Omega}|_{TM\times NM},\\[1mm]
\dd\omega^\perp + \omega^\perp\wedge\omega^\perp = \omega^{\II}\wedge\t\omega^{\II}+\widetilde{\Omega}|_{NM\times NM},
\end{cases}
\end{align}
where $\tau^*\,:\,T^*M\rightarrow T^*\widetilde{M}$ is the pullback of
the restriction map $\tau\,:\,u(M)\subset \widetilde{M}\rightarrow M$. The first equation is the Gauß equation, which relates to the {\it intrinsic} curvature given by the Riemann curvature tensor of $(M,g)$,
and the {\it extrinsic} curvature given by the second fundamental form.
The second equation is the Codazzi equation, and the final one is the Ricci equation, which in spirit is similar to the Gau\ss\ equation in which the {\it normal curvature}
$\dd\omega^\perp+\omega^\perp\wedge\omega^\perp$ is related to the second fundamental form. In general, system \eqref{system-gcr} is a mixed-type system of PDE (see \cite{bbg}).

We recall (see \cite{spivak}) that Cartan's lemma implies that $\omega^\II$ may be expressed in terms of the second fundamental form $\II$ by contracting with the tangent coframe,
\begin{equation}\label{defi-omega-II}
\omega^\II = \II\cdot\theta^\top.
\end{equation}
In index notation,
$$
(\omega^\II)^i_j = \II^i_{jk}\theta^k=\II^i_{kj}\theta^k \qquad\mbox{for $1\leq j,k\leq n=\dim(M)$ and $n+1\leq i\leq N=\dim(\widetilde{M})$}.
$$

For instance, in the case of a manifold $M$ immersed in a codimension-one manifold $\widetilde{M}$,
the equations in \eqref{system-gcr} are a reformulation of the classical Gau\ss--Codazzi equations,
\begin{align*}
&\II(X,Z)\II(Y,W)-\II(X,W)\II(Y,Z)=g(R(X,Y)Z,W)-\tilde{g}(\widetilde{R}(X,Y)Z,W),\\
&\nabla_Y\II(X,Z)-\nabla_X\II(Y,Z)=\widetilde{R}(X,Y)Z,
\end{align*}
while the Ricci equation is trivial. In the above two equations, $X,Y,Z$ and $	W$ are tangent vectors to $M$,
and we view $\widetilde{R}$ as being pulled back on $M$ for simplicity of notation.

We are now ready to present a precise statement of our main theorem for isometric immersions.

\begin{thm}\label{thm-main-isometric-immersions}\label{thm-weak-cty}
	Let $(M,g)$ be a Riemannian manifold.
	Let $u_\epsi$ be a sequence of weak $\W^{2,p}_\loc$ isometric immersions of $(M,g)$ into a Riemannian manifold $(\widetilde{M},\tilde{g})$, and
	let $\II_\epsi\in\L^p_\loc(M,\Hom(TM\times TM,NM))$ be the second fundamental forms associated
	to each $u_\epsi$ and satisfy that,  for any bounded domain $E\subset M$,
	\begin{equation}
	\sup_{\epsi>0}\Vert\II_\epsi\Vert_{\L^p(E)}<\infty \qquad\,\,\mbox{for $p>2$}.
	\end{equation}
	Then, up to a subsequence, $u_\epsi\rightharpoonup u$ weakly in $\W^{2,p}_{\loc}$ such that $u$ is a weak $\W^{2,p}_{\loc}$ isometric immersion
	and any orthonormal coframe on $(M,g)$ may be extended into a Darboux coframe on $(\widetilde{M},\tilde{g})$,
	adapted to $u$ and satisfying the Gau\ss--Codazzi--Ricci equations \eqref{system-gcr} in the sense of distributions.
\end{thm}

We remark that the Cartan equations are always verified by a sufficiently
regular coframe ($\theta\in\W^{1,2}$ is sufficient, see \eg \cite{phd-giron});
the point of the theorem is to construct the Darboux coframes along the sequence $u_\epsi$ so as to pass to limits
in the non-linear term
of the Gau\ss\, equation, which depends only on the intrinsic geometry of $M$.
Thus it represents a higher-order obstruction to the existence of isometric immersions.

	This result is related to the convergence theorems
	for Sobolev immersions with bounded second fundamental form and volume in {\rm \cite{langer,breuning}}, whose main result is the following:
	consider a sequence $u_\epsi$ of $\W^{2,p}$ immersions with $p>n$ from closed Riemannian manifolds $M_\epsi$
	into $(\R^N,\euc)$ such that
	\begin{align}
	\sup_{\epsi>0}\Vert\II_\epsi\Vert_{\L^p(M_\epsi)}+ \sup_{\epsi>0}\Vol(M_\epsi)<\infty, \label{langer-volume}
	\end{align}
	fixing a common point $q\in u_\epsi(M_\epsi)\subset\R^N$.
	Then there exist a manifold $M$, $C^1$ diffeomorphisms $\Psi_\epsi\,:\,M\rightarrow M_\epsi$, and a $\W^{2,p}$ immersion $u$ such that,
	up to a subsequence, $u_\epsi\circ\Psi_\epsi\rightharpoonup u$ in the $C^1$--topology.
	As an immediate consequence, the metric sequence $(u_\epsi\circ\Psi_\epsi)^*\euc$ induced by the Euclidean structure
	on $M$ converges uniformly.
	
	In our context, if the manifold is closed,
	given a sequence of isometric $\W^{2,p}$ immersions,
	the isometry condition clearly implies volume constancy.
	Thus, when $p>n$, the $C^1$--convergence of the immersions is clear from the main theorem of {\rm \cite{breuning}}{\rm ;}
	so is the fact that the limiting immersion is isometric.
	Therefore, the key point of Theorem { \ref{thm-weak-cty}}
	is about the convergence of \emph{curvatures} along sequences of isometric immersions for $p>2$ {\rm (}$p>n$ is not necessary{\rm )}.

\subsection{Div-curl structure}\label{subsection-method}
The PDEs involved present a div-curl structure
with respect to $\mathbf{Ad}(P)$-valued differential forms representing connections.
Our contribution is that this structure alone is sufficient to pass to limits in the non-linear terms of the equations considered. More precisely, our approach adopts the point of view of \cite{div-curl-rrt}, which showed that  differential forms
with suitably controlled exterior differentials satisfy some compensation properties.
We prove that, for a uniform bounded sequence of connections in $\L^p_\loc$, the curvature tensor passes to the limit
in the sense of distributions. This is the content of the next proposition, which is of independent interest.

\begin{prop}[Weak continuity of $\Omega$]\label{thm-main-chapter-connections}
	Let $(M,g)$ be a Riemannian manifold, let $P$ be a principal $G$-bundle over $M$ with $G$ as a matrix group,
	and let $\widetilde{\nabla}\in\mathcal{A}(P)$ be a
	reference connection.
	Assume that $\nabla_{A_\epsi}=\widetilde\nabla+A_\epsi$ is a sequence of connections on $P$, and
	$\Omega_{A_\epsi}$ is the associated sequence of curvature forms
	such that, for any bounded domain $E\subset M$,
	\begin{align*}
	&\sup_{\epsi>0}\Vert A_\epsi\Vert_{\L^p(E)}<\infty\qquad \mbox{for $p>2$},\\
	&\sup_{\epsi>0}\Vert \Omega_\epsi\Vert_{\mathcal{M}(E)}<\infty.
	\end{align*}
	Let $\nabla_A=\widetilde{\nabla}+A$ be an $\L^p_\loc$ connection such that $A_\epsi\rightharpoonup A$ in $\L^p_{\loc}$
	up to a non-relabelled subsequence.
	Then $\Omega_{A_\epsi}\weakstar \Omega_A$ in the sense of distributions:
	for any compactly supported $\phi\in\Lambda^2(M,\mathbf{Ad}(P))$,
	\begin{equation}
	\lim_{\epsi\rightarrow0}\int_M \big(\langle A_\epsi,\delta\phi\rangle + \frac{1}{2}\langle [A_\epsi\wedge A_\epsi],\phi\rangle\big)
	=\int_M \big(\langle A,\delta\phi\rangle + \frac{1}{2}\langle [A\wedge A],\phi\rangle\big).
	\end{equation}
\end{prop}

Proposition \ref{thm-main-chapter-connections} follows from a general div-curl lemma for differential forms (see Lemma \ref{lem-div-curl-lie-algebra} below)
dealing with Lie-algebra-valued forms; this extends the result in \cite{div-curl-rrt} which is only for scalar-valued forms.

To prove Theorems \ref{thm-yang-mills} and \ref{thm-main-isometric-immersions}, we combine this argument with the special structure of the equations:
in the case of the Yang--Mills equations,
we use the Bianchi identity:
\begin{equation*}
\DD_A\Omega_A=0.
\end{equation*}
For the Gau\ss--Codazzi--Ricci system, the same reasonning applies; the main difficulty in the proof of Theorem \ref{thm-main-isometric-immersions} consists in checking the validity
of the Gau\ss--Codazzi--Ricci equations in the limit, given only that
the second fundamental form sequence is uniformly bounded in $\L^p_\loc$ {\it a priori}, whereas Proposition \ref{thm-main-chapter-connections} requires the whole connection form to be controlled in $\L^p$. When $p=2$, an improvement of Proposition \ref{thm-main-chapter-connections} is required and is detailed in \S \ref{section-generalisation}.

\section{Weak Continuity of the Yang--Mills Equations: Proof of Theorem \ref{thm-yang-mills}}\label{section-yang-mills}

The main goal of this section is to prove Proposition \ref{thm-main-chapter-connections} and Theorem \ref{thm-yang-mills}. We first establish a general div-curl lemma, Lemma \ref{lem-div-curl-lie-algebra}, for $\mathfrak{g}$-valued differential forms, from which we deduce both aforementioned results. Lemma \ref{lem-div-curl-lie-algebra} extends the result of \cite{div-curl-rrt} to the context of $\mathfrak{g}$-valued forms.
We emphasize that the algebraic properties of $\mathfrak{g}$-valued differential forms here differ
from scalar-valued differential forms considered in \cite{div-curl-rrt}.
For example, for a scalar-valued differential form $\omega$, the alternating property implies that
$\omega\wedge\omega=0$. However, this is not true for $\mathfrak{g}$-valued differential forms unless the underlying Lie group $G$ is Abelian.
Thus, a proof of the general div-curl lemma needed for our subsequent development is provided.
Some well-known arguments dealing with  Hodge theory in the general context of elliptic complexes have been
relegated to an appendix for the sake of providing a complete presentation (see Appendix \ref{appendix-div-curl}).

\subsection{Div-curl lemma for $\mathfrak{g}$-valued differential forms}
We consider the complex of $\mathfrak{g}$-valued differential forms:
\begin{equation}
E(i)=\Gamma(M,\mathfrak{g}\otimes\wedge^i T^*M),
\end{equation}
and the differential given by $\DD(i)=\dd$, the exterior differential,
where $\mathfrak{g}$ is equipped with both the Lie bracket (which makes $E(i)$ into a non-associative bundle) and an inner product.

\begin{lem}\label{lem-div-curl-lie-algebra}
	Let $(M,g)$ be a
	Riemannian manifold with $n=\dim M$, let $1\leq \mu_i\leq n$ and $1< p_i< \infty$ such that
	\begin{align*}
	\sum_{i=1}^k \frac{1}{p_i}=1,\qquad \sum_{i=1}^k \mu_i=:s\leq n,
	\end{align*}
	and let $\mathfrak{g}$ be a Lie algebra.
	Assume that $A^i_\epsi\in\L^{p_i}_{\loc}(M,\mathfrak{g}\otimes \wedge^{\mu_i}T^*M)$ is a sequence
	of differential forms such that
	\begin{align}
	&A_\epsi^i\rightharpoonup A^i\qquad\text{weakly in $\L^{p_i}_\loc$ as $\epsi\rightarrow0$},\label{eqn-div-curl-compact-confinement-a}\\
	&\dd A_\epsi^i\Subset \W^{-1,p_i}_{\loc}.\label{eqn-div-curl-compact-confinement}
	\end{align}
	Then, for an arbitrary smooth, compactly supported test form $\phi\in\Gamma(M,\mathfrak{g}\otimes \wedge^{n-s}T^*M)$,
	\begin{equation}\label{3.4a}
	\lim_{\epsi\rightarrow0} \int_M\langle  [A_\epsi^1\wedge[\dots\wedge A_\epsi^k]\cdots], \phi\rangle
	= \int_M\langle [ A^1\wedge[\dots\wedge A^k]\cdots],\phi\rangle.
	\end{equation}
\end{lem}

\begin{proof} The proof is based on an induction argument. There is nothing to prove when $k=1$.

	\noindent
	\textbf{1.} We now prove \eqref{3.4a} for $k=2$.
	Let $A^1_\epsi$ and $A^2_\epsi$ be $\mathfrak{g}$-valued forms of degree $\mu_1$ and $\mu_2$, respectively.
	The complex $E(i)=\Gamma(M,\mathfrak{g}\otimes\wedge^iT^*M)$, equipped with the exterior differentiation $\dd$, is an elliptic complex.
	Thus, by Lemma \ref{lem-hodge-decomp}, for $p_1,p_2\in (1,\infty)$,
	for each $A_\epsi^i$ satisfying
	conditions \eqref{eqn-div-curl-compact-confinement-a}--\eqref{eqn-div-curl-compact-confinement},
	there exist $\Psi^i_\epsi$ and $\rho_\epsi^i$ such that
	\begin{align*}
	&A_\epsi^i = \dd\Psi_\epsi^i + \rho_\epsi^i,\\
	&\Psi^i_\epsi\rightarrow\Psi^i &\text{ in } \L^{p_i}_{\loc}\\
	&\dd\Psi^i_\epsi \rightharpoonup \dd\Psi^i &\text{ in } \L^{p_i}_{\loc}\\
	&\rho_\epsi^i \rightarrow\rho^i  &\text{ in } \L^{p_i}_{\loc}.
	\end{align*}
	Thus, as in the proof of Lemma \ref{lem-div-curl-step},
	every product $[A_\epsi^i\wedge A_\epsi^j]$ can be decomposed as the sum of a product
	of form $[\dd\Psi_\epsi^i\wedge \dd\Psi_\epsi^j]$, a product of a weakly-strongly converging sequence
	of form $[\dd\Psi_\epsi^i\wedge \rho_\epsi^j]$ (that converges in $\L^1_\loc$), and
	a strongly convergent product $[\rho_\epsi^i\wedge \rho_\epsi^j]$.
	Therefore, the only term we have to deal with is $[\dd\Psi_\epsi^1\wedge \dd\Psi_\epsi^2]$.
	
	Integrating by parts, we obtain that,
	for an arbitrary smooth, compactly supported test form $\phi\in\Gamma(M,\mathfrak{g}\otimes \wedge^{n-s}T^*M)$,
	\begin{align*}
	\int_M \langle [\dd\Psi_\epsi^1\wedge \dd\Psi_\epsi^2],\phi\rangle
	&= \int_M \langle \dd\Psi_\epsi^1,*[\dd\Psi_\epsi^2\wedge *\phi]\rangle
	= \int_M \langle \Psi_\epsi^1, \delta*[\dd\Psi_\epsi^2\wedge *\phi]\rangle\\
	&= \int_M \langle \Psi_\epsi^1,*\dd[\dd\Psi_\epsi^2\wedge *\phi]\rangle
	= \int_M \langle [\Psi_\epsi^1\wedge \dd\Psi_\epsi^2], \delta\phi\rangle,
	\end{align*}
	by virtue of Stokes' theorem and Jacobi's identity,
	where $\delta$ is the co-differential on $M$.
	Using that $\Psi^1_\epsi$ converges strongly to $\Psi^1$ in $\L^{p_1}_\loc$,
	we conclude that,
	\begin{align*}
	\int_M \langle [\Psi_\epsi^1\wedge \dd\Psi_\epsi^2], \delta\phi\rangle
	\rightarrow -\int_M\langle [\Psi^1\wedge \dd\Psi^2],\delta\phi\rangle
	= \int_M \langle [\dd\Psi^1\wedge \dd\Psi^2],\phi\rangle \qquad\mbox{as $\epsi\rightarrow0$},
	\end{align*}
	
	\medskip		
	\noindent
	\textbf{2}.
	We now deal with the general case.
	For $\mathfrak{g}$-valued forms $A^i, i=1,\dots,k$, each of degree $\mu_i$,
	assume by induction that the statement holds for the product of $k-1$ forms.
	By Lemma \ref{lem-hodge-decomp}, we may write $A^i_\epsi=\dd\Psi_\epsi^i + \rho^i_\epsi$ for all $i=1,\dots,k$.
	Thus, the product $[A_\epsi^1\wedge [\dots\wedge A_\epsi^k]\cdots]$ may be expressed as a sum of products
	involving $\dd\Psi_\epsi^i$ and $\rho_\epsi^i$.
	In this sum, the only term we have to treat is
	the product, $Q^\epsi:=[\dd\Psi_\epsi^1 \wedge [\dots\wedge \dd\Psi_\epsi^k]\cdots]$;
	all the other terms are treated by the induction hypothesis,
	the Banach--Alaoglu theorem, and the uniqueness of weak limits.
	
	Note that, for each fixed $\epsi$, the form $Q_\epsi$
	is closed,
	{\it i.e.}, $\dd Q_\epsi=0$.
	Integrating by parts, we obtain that,
	for an arbitrary smooth, compactly supported test form $\phi\in\Gamma(M,\mathfrak{g}\otimes \wedge^{n-s}T^*M)$,
	\begin{align*}
	\int_M \langle Q_\epsi, \phi\rangle
	&= \int_M \langle \dd\Psi_\epsi^1,*[\dd\Psi_\epsi^2\wedge [\dots\wedge \dd\Psi_\epsi^k]\cdots]\wedge *\phi]\rangle\\
	&= \int_M \langle [\Psi_\epsi^1\wedge [\dd\Psi_\epsi^2 \wedge [\dots\wedge \dd\Psi_\epsi^k]\cdots]], \delta\phi\rangle\\
	&\rightarrow -\int_M \langle [\Psi^1\wedge [\dd\Psi^2 \wedge [\dots\wedge \dd\Psi^k]\cdots]], \delta\phi\rangle\\
	&=\int_M \langle [\dd\Psi^1\wedge [\dd\Psi^2 \wedge [\dots\wedge \dd\Psi^k]\cdots]], \phi\rangle
	\end{align*}
	by virtue of Stokes' theorem and Jacobi's identity.
	
	By induction, we complete the proof. %
\end{proof}

\subsection{Weak continuity of $\Omega$ (proof of Proposition \ref{thm-main-chapter-connections})}	
Recall that, for two connections $\widetilde{\nabla}$ and $\nabla_A=\widetilde{\nabla}+A$, their curvatures are related by
\begin{equation*}
\Omega_A=\Omega_{\widetilde{\nabla}} + \widetilde{\nabla} A +\frac{1}{2}[A\wedge A].
\end{equation*}
Therefore, given a sequence of connections $\nabla_{A_\epsi}=\widetilde{\nabla}+A_\epsi$, verifying the convergence
of the sequence of the associated curvatures in the sense of distributions
amounts to showing the convergence of the quadratic term $[A_\epsi\wedge A_\epsi]$.

Let us remark that the complex formed by
\begin{equation}\label{eqn-ad-p-complex}
E(i)=\mathbf{Ad}(P)\otimes\wedge^iT^*M,\qquad \DD(i)=\widetilde{\nabla}.
\end{equation}
is not exact, unless $\widetilde{\nabla}\circ\widetilde{\nabla}=0$, {\it i.e.}, ${\Omega_{\widetilde{\nabla}}=0}$.
This in turns implies that $E(i)$ are trivial bundles. Thus one cannot apply directly Lemma \ref{lem-div-curl-lie-algebra} to the complex \eqref{eqn-ad-p-complex}; instead, a localization argument is necessary.

\medskip
\noindent

\begin{proof}[\bf Proof of Proposition {\rm \ref{thm-main-chapter-connections}}]
	We divide the proof into two steps.
	
	\noindent
	\textbf{1.} We first consider the local case on a ball $B\subset M$.
	Fix a local trivialization of bundle $P$ so that $P\simeq G\times B$ locally.
	The parameters of this trivialization induce a trivialization of $\mathbf{Ad}(P)$ as $\mathfrak{g}\times \R^n$,
	for the Lie algebra $\mathfrak{g}$ associated to the Lie group $G$ (recall that $P$ is a principal $G$-bundle).
	In this trivialization, we consider the trivial connection $\widetilde\nabla=\dd$,
	which acts by the exterior differentiation.
	Then the complex formed by $E(i)=\Gamma(B,\mathfrak{g}\otimes\wedge^i\R^n)$ and $\DD(i)=\dd$ is exact and elliptic.
	Moreover, the trivial bundle $\mathfrak{g}\times\wedge^i\R^n$ is an algebra bundle with internal multiplication
	given by the Lie bracket.
	
	With respect to $\widetilde\nabla$,
	we can write $\nabla_{A_\epsi}|_{B} = \dd+\tilde{A}_\epsi$
	and $\nabla_A|_{B}=\dd+\tilde{A}$.
	Since $A_\epsi\rightharpoonup A$ weakly in $\L^p_\loc$, then
	$$
	\tilde{A}_\epsi\rightharpoonup\tilde{A} \qquad \mbox{in $\L^p_\loc$}.
	$$
	Moreover, in this trivialization, we have
	\begin{align*}
	\Omega_{A_\epsi} &= \Omega_{\widetilde{\nabla}} + \dd\tilde{A}_\epsi + \frac{1}{2}[\tilde{A}_\epsi\wedge \tilde{A}_\epsi]
	= \dd\tilde{A}_\epsi + \frac{1}{2}[\tilde{A}_\epsi\wedge \tilde{A}_\epsi].
	\end{align*}
	
	It is clear that
	$$
	\dd\tilde{A}_\epsi\rightharpoonup
	\dd\tilde{A}\qquad\,\,\mbox{ in the sense of distributions.}
	$$
	Since $\dd\tilde{A}_\epsi= \Omega_\epsi-\frac{1}{2}[\tilde{A}_\epsi \wedge \tilde{A}_\epsi ]$,
	$\dd\tilde{A}_\epsi $ is in a bounded subset of Radon measures,
	which compactly embeds into $\W^{-1,q}_\loc$ for $q<1^*=\frac{n}{n-1}$:
	$$
	\dd\tilde{A}_\epsi \qquad \mbox{is compact in $\W^{-1,q}_{\loc}$ for $q<\frac{n}{n-1}$}.
	$$
	On the other hand, $\tilde{A}_\epsi $ is uniformly bounded in $\L^p$ so that
	$$
	\dd\tilde{A}_\epsi \qquad \mbox{is bounded in $\W^{-1,p}_{\loc}$}.
	$$
	By interpolation, as $q<\frac{n}{n-1}\leq2<p$, we obtain
	$$
	\dd\tilde{A}_\epsi \Subset\W^{-1,2}_{\loc}.
	$$
	By Lemma \ref{lem-div-curl-lie-algebra},
	we conclude that
	$$
	[\tilde{A}_\epsi\wedge\tilde{A}_\epsi]\rightharpoonup
	[\tilde{A}\wedge\tilde{A}]\qquad\,\,\mbox{in the sense of distributions.}
	$$
	
	Therefore, for a compactly supported $\phi\in\Gamma(B,\mathbf{Ad}(P)\otimes\wedge^2T^*M)$, we have
	\begin{align*}
	\int_M \langle \Omega_{A_\epsi},\phi\rangle = \int_B\langle \Omega_{A_\epsi},\phi\rangle
	&
	= -\int_B \big(\langle \tilde{A}_\epsi,\delta\phi\rangle +\frac{1}{2} \langle [\tilde{A}_\epsi\wedge\tilde{A}_\epsi],\phi\rangle\big)
	\\&
	\rightarrow \int_B \langle \Omega_A,\phi\rangle = \int_M \langle \Omega_A,\phi\rangle.
	\end{align*}
	
	\medskip		
	\noindent
	\textbf{2}. We now globalize the argument.
	To do this, we fix a bounded domain $E\subset M$ which we cover by local trivializable subsets $B^i$,
	and consider trivializations $P|_{B^i}$ on $B^i$;
	similarly for bundle $\mathbf{Ad}(P)\otimes T^*M$.
	By compactness, we can extract a finite subcover of local trivializations. The local argument applies on each of them.
	
	Let $\{\chi^i\}_{i=1}^m$  be a partition of unity adapted to the local trivializations
	so that $\sum_{i=1}^m\chi^i=1$ and $\supp\chi^i\subset B^i$.
	For a smooth $\phi\in\Lambda^2(M,\mathbf{Ad}(P))$ compactly supported in $E$,
	we have
	\begin{align*}
	\int_M \langle \Omega_{A_\epsi},\phi\rangle &=\sum_{i=1}^m \int_{B^i}\chi^i	\langle \Omega_{A_\epsi},\phi\rangle\\
	&=\sum_{i=1}^m \int_{B^i}\chi^i\big(-\langle \tilde{A}_\epsi,\delta\phi\rangle +\frac{1}{2} \langle [\tilde{A}_\epsi\wedge\tilde{A}_\epsi],\phi\rangle\big)\\
	&\rightarrow \sum_{i=1}^m\int_{B^i}\chi^i \langle \Omega_A,\phi\rangle = \int_M\langle \Omega_A,\phi\rangle.
	\end{align*}
	This completes the proof.
\end{proof}
\begin{rk}\label{rk-asd}
	Applying this result to the anti-self-duality equation: $*\Omega_A=\omega\wedge\Omega_A$, it is straightforward to
	check that anti-self-dual connections are stable under weak $\L^p$-convergence for $p>2$.
\end{rk}

\medskip	
\subsection{Proof of Theorem \ref{thm-yang-mills}}\label{subsection-proof-yang-mills}
The main point in the proof of Theorem \ref{thm-yang-mills} is to verify the weak continuity of quadratic terms appearing in the Yang--Mills equations. Those terms are of the form
\begin{equation}\label{eqn-yang-mills-non-linear-term}
[\Omega_{A_\epsi}\wedge A_\epsi],
\end{equation}
\ie, that the non-linear terms pass to the limit in the sense of distributions.
To this effect, we use a weak version of the Bianchi identity (Step 1 of the proof).

\medskip
\begin{proof}[\bf Proof of Theorem {\rm \ref{thm-yang-mills}}]
	We divide the proof into four steps.
	
	\smallskip
	\noindent
	\textbf{1.}
	Let $p\geq\frac{3n}{n+2}$ and $\nabla_A=\widetilde{\nabla}+A\in\mathcal{A}^{1,p}_{\loc}(P)$,
	and let $\Omega_A$ be the curvature form. It is easy to see that
	\begin{equation}\label{lem-bianchi}
	\DD_A\Omega_A=0  \qquad \,\, \mbox{in the sense of distributions{\rm ;}}
	\end{equation}
	that is, for any smooth, compactly supported $\phi\in \Lambda^3(M,\mathbf{Ad}(P))$,
	\begin{equation}\label{eqn-the-more-precise-one}
	\int_M \langle \Omega_A, \delta\phi -*[A\wedge *\phi]\rangle=0.
	\end{equation}
	Indeed, recall that the curvature form $\Omega_A$ of a smooth connection $A\in\mathcal{A}(P)$ satisfies the Bianchi identity $\DD_A\Omega_A=0$.
	For $A\in\mathcal{A}^{1,p}_{\loc}(P)$,  $\Omega_A\in \L^{\frac{p^*}{2}}_{\loc}(M,\mathbf{Ad}(P)\otimes \wedge^2T^*M)$ so that, by Young's inequality,
	$$
	[\Omega_A\wedge A]\in \L^{q}_{\loc}(M,\mathbf{Ad}(P)\otimes \wedge^3T^*M)\qquad\,\, \mbox{for $\frac{1}{q}=\frac{1}{p}+\frac{2}{p^*}$},
	$$
	where $p^*=\frac{np}{n-p}$ is the Sobolev conjugate of $p$.
	
	To ensure the integrability of the product, we need $q \geq 1$, which is equivalent to the condition
	$$
	p\geq\frac{3n}{n+2}.
	$$
	Thus, taking a sequence of smooth connections approximating $A$ in $\mathcal{A}^{1,p}_{\loc}$, we can pass to the limit in equation \eqref{eqn-the-more-precise-one}.\\
	
	\smallskip
	\noindent
	\textbf{2.}
	Consider
	the sequence, $\{A_\epsi\}_{\epsi>0}$, satisfying $\DD_{A_\epsi}^*\Omega_{A_\epsi}=0$ in the sense of distributions.
	Then, for any smooth, compactly supported $\phi\in\Lambda^1(M,\mathbf{Ad}(P))$,
	$A_\epsi$ satisfy
	\begin{equation}\label{3.12a}
	\int_M \langle \Omega_{A_\epsi},\dd\phi\rangle
	=- \int_M \langle \Omega_{A_\epsi},[A_\epsi \wedge \phi]\rangle
	\end{equation}
	(also see equation \eqref{eqn-yang-mills-weak}).
	By Proposition \ref{thm-main-chapter-connections},
	the left-hand side of \eqref{3.12a}
	passes to the limit in the sense of distributions.

	\smallskip		
	\noindent\textbf{3.}
	We now focus on the right-hand side term.
	Let us fix a bounded domain $E$ such that $\phi$ is compactly supported inside $E$. By the triple product identity, for $\phi\in\Lambda^1(M,\mathbf{Ad}(P))$, we have
	\begin{align*}
	\int_E\langle \Omega_{A_\epsi},[A_\epsi\wedge \phi]\rangle &= \int_E \langle *[*\Omega_{A_\epsi}\wedge A_\epsi],\phi\rangle.
	\end{align*}
	Then the question reduces to
	the weak continuity of the non-linear term $[*\Omega_{A_\epsi}\wedge A_\epsi]$.
	
	By the Bianchi identity (see equation \eqref{lem-bianchi}), we have
	$$
	\int_M \langle \Omega_{A_\epsi}, \delta\psi \rangle
	=\int_M \langle \Omega_{A_\epsi},*[A_\epsi\wedge*\psi]\rangle
	$$
	for any compactly supported $\psi\in  \Lambda^3(M,\mathbf{Ad}(P))$.
	By Young's inequality, we find that $\dd\Omega_{A_\epsi}\in\L^1(E)$ and
	$$
	\Vert \dd\Omega_{A_\epsi}\Vert_{\L^1(E)}
	\leq  \Vert \Omega_{A_\epsi}\Vert_{\L^{2}(E)}\Vert A_\epsi\Vert_{\L^2(E)}
	\leq C \Vert \Omega_{A_\epsi}\Vert_{\L^2(E)}\Vert A_\epsi\Vert_{\L^4(E)},
	$$
	so that $\{\dd\Omega_{A_\epsi}\}$ is uniformly bounded in $\L^1(E)$, which implies that
	\begin{equation}\label{3.14a}
	\{\dd\Omega_{A_\epsi} \}\quad \mbox{is compactly contained in $\W^{-1,r}(E)$ for $r<1^*=\frac{n}{n-1}$},
	\end{equation}
	by the Rellich--Kondrachov theorem.
	On the other hand, the curvature form sequence is uniformly bounded in $\L^2$,
	so that $\dd\Omega_{A_\epsi}$ is uniformly bounded in $\W^{-1,q}_\loc$ for $q\leq2$.
	By interpolation, $\dd\Omega_{A_\epsi}$ is compactly contained in $\W^{-1,\frac{3}{2}}_{\loc}$.
	
	Furthermore, by definition, we have
	$$
	\dd A_\epsi= \Omega_{A_\epsi}-\frac{1}{2}[A_\epsi\wedge A_\epsi],
	$$
	so that $\dd A_\epsi$ is in a bounded subset of $\L^{\frac{3}{2}}(E)$
	which implies that
	\begin{equation}\label{3.16a}
	\{\dd A_\epsi \}\quad \mbox{is compact in $\W^{-1,s}(E)$ for $s<\frac{3n}{2n-3}$}.
	\end{equation}
	As $A_\epsi$ is uniformly bounded in $\L^4_{\loc}$,
	\begin{equation}\label{3.17a}
	\{\dd A_\epsi\} \quad \mbox{is bounded in $\W^{-1,4}_{\loc}$}.
	\end{equation}
	By interpolation, we combine \eqref{3.16a} with \eqref{3.17a} to obtain
	\begin{equation}\label{3.18a}
	\{\dd A_\epsi\}\Subset\W^{-1,3}(E).
	\end{equation}
	
	\noindent
	{\bf 4}. We are now in a position to apply Lemma \ref{lem-div-curl-lie-algebra} to the product $[\Omega_{A_\epsi}\wedge A_\epsi]$
	with $p_1=\frac{3}{2}$, $p_2=3$, and $k=2$.
	Using \eqref{3.14a} and \eqref{3.18a}, we conclude the proof.
\end{proof}

\begin{rk}
	The proof carries in the same way if one considers approximate solutions of the Yang--Mills equations,
	{\it i.e.}, connections such that there exists a sequence of one-forms $\eta_\epsi$ such that
	$\DD_{A_\epsi}^*\Omega_{A_\epsi}=\eta_\epsi$ and $\eta_\epsi\rightarrow0$ in $\mathcal{D}'$.
\end{rk}

\section{Weak Continuity of the Gau\ss--Codazzi--Ricci Equations: Proof of
	Theorem \ref{thm-weak-cty}}\label{section-immersions}\label{section-defi-immersion}

The basic setting we consider is a Riemannian manifold $(M,g)$
and a sequence of immersions of Riemannian manifolds $u_\epsi\,:\,M\rightarrow\R^N$ satisfying
the Gau\ss--Codazzi--Ricci equations.
The arguments and proofs of this section are essentially unchanged if $(\R^N,\euc)$ is replaced by an arbitrary smooth
Riemannian manifold $(\widetilde{M},\tilde{g})$. For clarity of exposition, we first focus on the Euclidean case $(\R^N,\euc)$ in §\ref{subsection-proof-weak-cty}. The modifications required to handle the general case of an arbitrary background manifold are presented in §\ref{subsection-proof-weak-cty-ambient}. Subsection §\ref{subsection-darboux} contains some preliminary lemmas.

\subsection{Construction of the Darboux coframes}\label{subsection-darboux}
Recall that a choice of Riemannian metrics fixes an identification between the tangent and cotangent bundle.
In this manner, to each tangent vector $X$ is associated a dual one-form $\theta$ via the Riemannian metric $g$;
similarly, to each vector $u_*X$ is associated an extension of $\theta$, denoted by $\tilde\theta$.
Then, at a point $x\in M$,
we have
\begin{align}
&\theta(Y) = g(X,Y)\qquad\,\,\, \mbox{for any $Y\in TM_x$},\\
&\tilde\theta(Z) = \tilde{g}( u_*X,Z)\qquad \mbox{for any $Z\in (T\widetilde{M})_{u(x)}$}.
\end{align}
The isometry condition on $u$ implies that $|\theta|=|\tilde\theta|$.
Note in particular that $\tilde\theta\in \big(T^*\widetilde{M}\big)^\top$ since,
by definition,
$$
\tilde\theta(Z)=\tilde{g}( u_*X,Z)\ =0 \qquad\,\, \mbox{for any $Z\in \big(T\widetilde{M}\big)^\perp$}.
$$
Let
\begin{align*}
\tau^*\,:\,T^*M&\rightarrow T^*\widetilde{M}|_{u(M)},\\
\theta&\mapsto \tilde\theta.
\end{align*}
Similarly, we denote its dual map $\tau\,:\,T\widetilde{M}|_{u(M)}\rightarrow TM$,
which is the orthogonal projection onto the tangent space, satisfying
\begin{equation}
\theta(X^\top)=\theta(\tau X)=(\tau^*\theta)(X)\qquad\,\,\mbox{for $\theta\in T^*M_x$ and $X\in (T\widetilde{M})_{u(x)}$}.
\end{equation}
Standard properties of the quotients imply that there exist bundle morphisms $\nu\,:\,NM\rightarrow T\widetilde{M}$
and $\pi\,:\,T\widetilde{M}\rightarrow NM$ (the latter being the quotient map)
such that the following short sequences are exact:
\begin{align}	\label{eqn-ses-1}
0\rightarrow TM\xrightarrow{u_*}\,&TN|_{u(M)}\xrightarrow{\pi} NM\simeq T\widetilde{M}|_{u(M)}/TM\rightarrow0,\\
0\rightarrow NM \xrightarrow{\nu}\,&T\widetilde{M}|_{u(M)}\xrightarrow{\tau} TM\rightarrow 0.\label{eqn-ses-2}
\end{align}
If $u\,:\,(M,g)\rightarrow (\widetilde{M},\tilde{g})$ is isometric ({\it i.e.}, $u_*$ is an isometry onto its image)
and $NM$ is equipped with the pull-back metric $\nu^*\tilde{g}$, then all four maps $\nu,\tau,u_*$,
and $\pi$ are isometries.

Finally, for notational reference, we write out the dualized versions of the sequences \eqref{eqn-ses-1} and \eqref{eqn-ses-2}:
\begin{align}
0\rightarrow T^*M \xrightarrow{\tau^*} T^*\widetilde{M}|_{u(M)}\xrightarrow{\nu^*}N^*M\rightarrow0,\\
0\rightarrow N^*M\xrightarrow{\pi^*}T^*\widetilde{M}|_{u(M)}\xrightarrow{ u^*}T^*M\rightarrow0.
\end{align}

We now describe how to obtain bounds on mappings $\tau$ and $\nu$.

\begin{lem}\label{lem-estimate-tau}
	Let $u:(M,g)\rightarrow(\widetilde{M},\tilde{g})$ be a weak $W^{2,p}$ isometric immersion, $p\geq1$, and let $E$ be a bounded parallelisable domain of $M$.
	Then there exists a constant $C>0$ depending only on the $C^{0,1}$ norm of $g$ and the $C^{1,1}$ norm of $\tilde{g}$ such that, for any one-form $\theta\in T^*M$ on $E$,
	\begin{equation}\label{eqn-boundedness-tau}
	\Vert \tau^*\theta\Vert_{\W^{1,p}(E)}
	\leq C\big(\Vert\nabla \theta\Vert_{\L^p(E)}\Vert\nabla u\Vert_{\L^\infty(E)} + \Vert \theta\Vert_{\L^\infty(E)}\Vert\nabla^2u\Vert_{\L^p(E)}\big).
	\end{equation}
	Similar bounds also hold for $\nu^*$.
\end{lem}

\begin{proof}
	Let $\theta^\sharp=X\in TM$, let $\tilde\theta=\tau^*\theta$,
	and let $Z$ and $W$ be local vector fields defined in a neighborhood of $u(M)$. We calculate
	\begin{equation*}
	\nabla_W\tilde\theta(Z) = \tilde{g}( u_*X,\nabla_WZ	)+\tilde{g}(\nabla_W(u_*X),Z),
	\end{equation*}
	where $\nabla$ denotes here the Levi-Civita connection induced by $\tilde{g}$. By a crude application of Young's inequality, recalling that $u$ is a bi-Lipschitz homeomorphism, we obtain
	\begin{align*}
	\Vert \nabla\tilde\theta\Vert_{\L^p}
	&\leq C(|\tilde{g}|_{C^{1,1}})\big(\Vert\nabla X\Vert_{\L^p}\Vert\nabla u\Vert_{\L^\infty} + \Vert X\Vert_{\L^\infty}\Vert\nabla^2u\Vert_{\L^p}\big)\\
	&\leq C(|\tilde{g}|_{C^{1,1}},|g|_{C^{0,1}})\big(\Vert\nabla \theta\Vert_{\L^p}\Vert\nabla u\Vert_{\L^\infty} + \Vert \theta\Vert_{\L^\infty}\Vert\nabla^2u\Vert_{\L^p}\big),
	\end{align*}
	where $C$ depends only on $g$ and the first derivative of $g$.
\end{proof}

We now extend these operators to frame bundles.
Let $X\in\mathscr{D}(u(M))$ be an adapted Darboux frame.
In a local trivialization, we may write $X=(X^1,\dots,X^N)$.
Then $\tau X=(X^1,\dots,X^n)=X^\top$ is a tangent frame on $TM$.
Similarly, $\pi X=(X^{n+1},\dots,X^N)=X^\perp$ is a normal frame in $NM$.
As usual, we abuse notation and denote by the same symbol $\tau$ (respectively $\pi$)
the map on frames $\mathscr{F}\tau$ (respectively $\mathscr{F}\pi$).

By duality, for a tangent coframe $\alpha=(\alpha^1,\dots,\alpha^n)$ in $T^*M$
and a normal coframe $\eta=(\eta^1,\dots,\eta^{N-n})$ in $N^*M$,
\begin{equation*}
\tau^*\alpha+\pi^*\eta = (\tau^*\alpha^1,\dots,\tau^*\alpha^n,\pi^*\eta^1,\dots,\pi^*\eta^{N-n})
\end{equation*}
is an adapted coframe.

\begin{lem}\label{lem-generating-darboux}
	Let $(M,g)$, $(\widetilde{M},\tilde{g})$ and $u$ be as in Lemma \ref{lem-estimate-tau},
	and let $TM$ and $NM$ be equipped with the metrics induced by $\tilde{g}$.
	Then any choice of local orthonormal coframe $\alpha$ on $T^*M$ and $\eta$ on $N^*M$ yields
	a local Darboux coframe $\theta\in\mathscr{D}^*(u(M))$.
	Moreover, for any bounded domain $E\subset M${\rm :}
	\begin{align}\label{eqn-estimate-darboux}
	\Vert\theta\Vert_{\W^{1,p}(E)}
	\leq C\big((\Vert \alpha\Vert_{\W^{1,p}(E)} + \Vert\eta\Vert_{\W^{1,p}(E)})\Vert\nabla u\Vert_{\L^\infty(E)}
	+ \Vert\nabla u\Vert_{\W^{1,p}(E)}\big).
	\end{align}
\end{lem}

\begin{proof}
	The coframe $\theta$, when viewed as a coframe on $T^*\widetilde{M}$, is given by $\theta = \tau^*\alpha + \pi^*\eta$.
	Then $\theta$ is a Darboux coframe; in particular,
	$$
	{|\alpha^i|=|\tau^*\alpha^i|=1=|\eta^j|=|\pi^*\eta^j|}
	\qquad\mbox{for $i=1,\dots, n$ and $j=1,\dots,N-m$}.
	$$
	Moreover, Lemma \ref{lem-estimate-tau} applied to $\tau^*$ and $\nu^*$ implies \eqref{eqn-estimate-darboux}.
\end{proof}

\subsection{Proof of Theorem \ref{thm-weak-cty} in the flat case}\label{subsection-proof-weak-cty}

With the preparation of \S 4.1--\S 4.3, we can now provide the proof of Theorem \ref{thm-weak-cty}. For simplicity of notation, we consider the case $(\widetilde{M},\tilde{g})=(\R^N,\euc)$.

\begin{proof}
We divide the proof into four steps.

\medskip		
\noindent
{\bf 1}. Let $u_\epsi$ be a sequence of weak $\W^{2,p}_\loc$ isometric immersions of $(M,g_\epsi)$ in $\R^N$ for $p>2$.
Then
there is a sequence of maps $\nu_\epsi\,:\,NM\rightarrow T\R^N$, which induces
a sequence of metrics $\hat{g}_\epsi=\nu_\epsi^*\euc\in\L^\infty_{\loc}\cap\W^{1,p}_{\loc}$ on the normal bundle $NM$.

As $u_\epsi\rightharpoonup u$ in $\W^{2,p}_{\loc}\cap \W^{1,\infty}_{\loc}$,
there exists a map $\nu\in\W^{1,p}_{\loc}\cap\L^\infty_{\loc}$ such that $\nu_\epsi\rightharpoonup\nu$ in $\W^{1,p}_{\loc}\cap \L^\infty_{\loc}$,
which induces a metric $\hat{g}=\nu^*\euc$ on $NM$.

Let $E\subset M$ be a bounded domain, and let $\alpha$ be an orthonormal tangent coframe on $TM|_E$ with respect to $g$.
By hypothesis, $g_\epsi\rightarrow g$. Then orthonormalizing $\alpha$ with respect to each $g_\epsi$ yields
a sequence $\alpha_\epsi$ of coframes on $TM$ such that $\alpha_\epsi$ is orthonormal with respect to $g_\epsi$
and $\alpha_\epsi\rightarrow \alpha$.

If necessary by restricting $E$ to a smaller domain,
we fix a trivialization of $NM$ over a neighborhood $E\subset M$.
This trivialization fixes coordinate vectors $\{\partial^{n+1},\dots,\partial^N\}$ on $NM$.
Orthonormalizing this basis with respect to $\hat{g}_\epsi$ yields a sequence of orthonormal coframes in $NM$, say $\eta_\epsi$,
such that $\eta_\epsi\rightharpoonup\eta$ in $\W^{1,p}_{\loc}\cap \L^\infty_{\loc}$.

\medskip
\noindent
{\bf 2.} We now define the following sequence of the Darboux coframes. For $\epsi>0$, let
\begin{equation*}
\theta_\epsi = \tau_\epsi^*\alpha_\epsi+\pi_\epsi^*\eta_\epsi,
\qquad
\theta=\tau^*\alpha+\pi^*\eta.
\end{equation*}
By Lemma \ref{lem-generating-darboux}, it is clear that $\theta_\epsi\rightharpoonup\theta$
in $\W^{1,p}_{\loc}\cap L^\infty_{\loc}$.
Let
\begin{equation}\label{4.21a}
\dd\theta_\epsi=\omega_\epsi\wedge\theta_\epsi \qquad \mbox{for $\epsi>0$}.
\end{equation}
Recall that, for the Darboux coframes, the connection form $\omega_\epsi$ splits into a tangential
connection form $\omega^\top_\epsi$,
a normal connection form $\omega_\epsi^\perp$,
and a second fundamental form $\omega^\II_\epsi$.
Then we claim that, for any bounded domain $E\subset M$,
\begin{equation}\label{4.20a}
\sup_{\epsi>0}\Vert(\omega_\epsi^\top, \omega_\epsi^\II,\omega_\epsi^\perp)\Vert_{\L^{p}(E)}<\infty.
\end{equation}

This follows from the general fact that a sequence of uniformly bounded orthonormal coframes in $\W^{1,p}$ locally
generates a sequence of uniformly bounded connection forms in $\L^p_\loc$.
More precisely, let $\theta_\epsi=(\theta_\epsi^\top, \theta_\epsi^\perp)$ be a sequence of the Darboux coframes indexed by $\epsi>0$.
Then system \eqref{4.21a}
is linear in $\omega_\epsi$.
Inverting and recalling that $|\theta_\epsi|=1$ since the coframe is orthonormal by definition,
we conclude that, for any bounded domain $E\subset M$,
$$
\Vert\omega_\epsi\Vert_{\L^p(E)}\leq C\Vert\theta_\epsi\Vert_{\W^{1,p}(E)}.
$$

\noindent
{\bf 3.} From Step 2, we conclude from
the weak convergence and linearity
that
$$
\omega_\epsi\rightharpoonup\omega \qquad \mbox{weakly in $\L^p_{\loc}$}.
$$
Moreover, for each $\epsi>0$,
\begin{equation}\label{eqn-cartan-epsi}
\dd\omega_\epsi + \omega_\epsi\wedge\omega_\epsi=0.
\end{equation}
We conclude by the weak continuity of curvature (Proposition \ref{thm-main-chapter-connections})
that
\begin{equation}\label{eqn-cartan}
\dd\omega+\omega\wedge\omega=0.
\end{equation}
Recall that
\begin{equation}
\omega=\begin{pmatrix}
\omega^\top & -\omega^\II\\[1mm]
\t\omega^\II &\omega^\perp
\end{pmatrix}\in\so(N);
\end{equation}
as the Cartan equations \eqref{eqn-cartan} are satisfied, the Codazzi and Ricci equations are thus satisfied in the limit.

\noindent
{\bf 4.} We now claim that $\omega^\top = \tau^*\beta$, where $\beta$ is the connection form of coframe $\alpha$,
{\it i.e.}, $\dd\alpha=\beta\wedge\alpha$,
which follows from a uniqueness argument.
Then we have
\begin{align*}
\dd\theta^\top|_{TM} = \dd(\tau^*\alpha)= \tau^*\dd\alpha
= \tau^*(\beta\wedge\alpha)
=\tau^*\beta\wedge \theta^\top|_{TM}.
\end{align*}

Recall that $\alpha$ is a coframe on $(M,g)$ and $g$ is a $C^2$ metric,
so that the connection form $\beta$ satisfies the curvature equation
\begin{equation}
\dd\beta+\beta\wedge\beta=\Rm_g(\alpha\wedge\alpha)\qquad\,\,\mbox{in the sense of distributions},
\end{equation}
where the right-hand side is understood
as the curvature endomorphism acting on two-forms.
Thus, we recover the Gau\ss{} equation:
\begin{align*}
\t\omega^\II\wedge\omega^\II = \dd\omega^\top + \omega^\top\wedge\omega^\top
= \tau^*(\dd\beta+\beta\wedge\beta)
= \tau^*\Rm_g(\alpha\wedge\alpha).
\end{align*}
This completes the proof.
\end{proof}

\subsection{Proof of Theorem \ref{thm-weak-cty} for arbitrary ambient manifolds}\label{subsection-proof-weak-cty-ambient}

As mentioned earlier, there is nothing special about $\R^N$ in Theorem \ref{thm-weak-cty}, and thus the proof may be given for a general ambient Riemannian manifold $(\widetilde{M},\tilde{g})$.

As before, we denote by $\mathbf{Rm}$ the Riemann curvature tensor of $g$,
and $\widetilde{\mathbf{Rm}}$ that of $\tilde{g}$.
Similarly, the curvature form of an orthonormal coframe on $(M,g)$ is denoted by $\Omega$,
and that of an orthonormal coframe on $(\widetilde{M},\tilde{g})$ by $\widetilde\Omega$.

\begin{proof}
	Let $u_\epsi$ be a sequence of isometric immersions $u_\epsi\,:\,(M,g)\rightarrow(\widetilde{M},\tilde{g})$.
	Let $E\subset M$ be any bounded domain, and let $\alpha$ be an orthonormal coframe on $TM|_E$.
	Restricting domain $E$ if necessary,
	we may consider a local trivialization of $NM|_E$.
	As in the proof of Theorem \ref{thm-weak-cty},
	maps $\nu_\epsi$ induce the metrics on $NM$ by pull-back, namely $\widehat{g_\epsi}=\nu_\epsi^*\tilde{g}$.
	Thus, we find as earlier that there is a sequence of coframes $\eta_\epsi$ defined on $N^*M$,
	orthonormal with respect to $\widehat{g}_\epsi$.
	
	Moreover, by weak convergence, there is a map $\nu\,:\,M\rightarrow \widetilde{M}$,
	a metric $\hat{g}=\nu^*\tilde{g}$, and a coframe $\eta$ orthonormal with respect to $\hat{g}$
	such that $\eta_\epsi\rightharpoonup\eta$ weakly in $\W^{1,p}(E)$ and weak-* in $L^\infty$.
	
	We define the Darboux coframe
	$$
	\theta_\epsi=\tau_\epsi^*\alpha+\pi_\epsi^*\eta_\epsi \qquad \mbox{for $\epsi>0$}.
	$$
	
	Let $\dd\theta_\epsi=\omega_\epsi\wedge\theta_\epsi$ for $\epsi>0$.
	The weak convergence of $\theta_\epsi\rightharpoonup\theta$ and linearity
	imply that $\omega_\epsi\rightharpoonup\omega$ weakly in $\L^p(E)$.
	
	By construction, $\theta_\epsi$ is a $\W^{1,p}(E)$ coframe on $(\widetilde{M},\tilde{g})$ so that it must satisfy the curvature equation:
	\begin{equation}
	\dd\omega_\epsi + \omega_\epsi\wedge\omega_\epsi = \widetilde{\Rm}(\theta_\epsi\wedge\theta_\epsi).
	\end{equation}
	Recalling that $\theta_\epsi$ is orthonormal and uniformly bounded in $\L^\infty(E)$,
	the curvature term $\widetilde{\Rm}(\theta_\epsi\wedge\theta_\epsi)$ is uniformly bounded in $\L^\infty(E)$.
	Thus, we find that,
	for $\epsi>0$, the term
	$${-\omega_\epsi\wedge\omega_\epsi+\widetilde{\Rm}(\theta_\epsi\wedge\theta_\epsi)}$$
	is uniformly bounded in $\L^{\frac{p}{2}}(E)$,
	which implies that
	$$
	\{\dd\omega_\epsi\}
	\qquad \mbox{embeds compactly into $\W^{-1,q}(E)$ for $q<\big(\frac{p}{2}\big)^*=\frac{np}{2n-p}$}.
	$$
	As $\omega_\epsi$ is uniformly bounded in $\L^p(E)$,
	$$
	\{\dd\omega_\epsi\}  \qquad\mbox{is uniformly bounded in $\W^{-1,p}(E)$}.
	$$
	Then, by interpolation,
	$\dd\omega_\epsi$ is compactly contained in $\W^{-1,2}(E)$.
	We conclude, as in
	Proposition \ref{thm-main-chapter-connections}, that
	\begin{equation}
	\dd\omega+\omega\wedge\omega=\widetilde{\Rm}(\theta\wedge\theta).
	\end{equation}
	Decomposing the connection form $\omega$ along its tangential, normal, and $\II$ part yields the Gau\ss--Codazzi--Ricci equations,
	as in the proof of Theorem \ref{thm-weak-cty}, which are therefore satisfied in the limit as ${\epsi\rightarrow0}$.
\end{proof}

\section{Extensions of Theorem \ref{thm-weak-cty}}\label{section-generalisation}

In this section, we present extensions of Theorem \ref{thm-weak-cty}. First, the case ${p=2}$ is studied in \S\ref{subsection-proof-weak-cty-equi}. In §\ref{subsection-minimisers}, Theorem \ref{thm-weak-cty} is applied to the problem of finding isometric immersions minimising some energy functional involving the second fundamental form. Finally, we mention another extension of Theorem \ref{thm-weak-cty} to semi-Riemannian metrics in \S\ref{subsection-semi}.

\subsection{Borderline case $p=2$}\label{subsection-proof-weak-cty-equi}

A stronger result, including the case $p=2$, can be derived in the case of codimension-one.

\begin{thm}\label{thm-weak-cty-equi}
	Let $(M,g)$ be an $n$-dimensional Riemannian manifold,
	and let $u_\epsi\,:\,M\rightarrow\R^{n+1}$
	be a sequence of weak isometric immersions.
	Assume that $u_\epsi\rightharpoonup u$ weakly in $\W^{2,2}_{\loc}$.
	Then $u$ satisfies $\langle \dd u,\dd u\rangle =g$ and, on an arbitrary parallelizable subset of $M$,
	any orthonormal coframe on $(M,g)$ may be extended into an orthonormal Darboux coframe
	satisfying the Gau\ss--Codazzi equations.
\end{thm}

Compared to Theorem \ref{thm-weak-cty}, the main difficulty for the proof of Theorem \ref{thm-weak-cty-equi}
is to verify that the quadratic term $\t\omega^\II\wedge\omega^\II$, which is only bounded in $\L^1_\loc$ {\it a priori},
is in fact weakly continuous in the sense of distributions.
In this case, Lemma \ref{lem-div-curl-lie-algebra} fails to provide the desired conclusion.
Instead, we rely on Lemma \ref{lem-div-curl-equi} below, motivated by  Conti-Dolzmann-M\"{u}ller \cite{conti-div-curl}
for the classical case in Euclidean spaces.

\begin{lem}\label{lem-div-curl-equi}
	Let $(M,g)$ be a Riemannian manifold, $n=\dim M$, $1\leq \mu_i\leq n$, and {$1\leq p_i\leq \infty$} so that
	\begin{align*}
	\sum_{i=1}^k \frac{1}{p_i}=1,\qquad	\sum_{i=1}^k \mu_i:=s\leq n.
	\end{align*}
	Assume that $A^i_\epsi\in\L^{p_i}_{\loc}(M,\mathfrak{g}\otimes \wedge^{\mu_i}T^*M)$ be sequences of differential forms
	such that, as $\epsi\rightarrow0$,
	\begin{align}
	&A_\epsi^i\rightharpoonup A^i\qquad\text{weakly in }\L^{p_i}_{\loc},\\
	&\dd A_\epsi^i\Subset \W^{-1,1}_{\loc}(M, \mathfrak{g}\otimes \wedge^{\mu_i} T^*M),\\
	&[A^1_\epsi\wedge[A_\epsi^2\wedge[\dots\wedge A_\epsi^k]\cdots]] \quad\text{ is locally equi-integrable.}
	\end{align}
	Then
	$$
	A_\epsi^1\wedge\dots\wedge A_\epsi^k\, \longrightarrow\, A^1\wedge\dots\wedge A^k\qquad\mbox{in the sense of distributions}.
	$$
\end{lem}

\begin{proof}  %
	Without loss of generality, we assume  that $M$ is a bounded manifold; otherwise, we can prove the lemma on any bounded domain in $M$.
	We first deal with the case where $k=2$ and $\mu_1=\mu_2=1$.
	We divide the proof into three steps.
	
	\smallskip
	\noindent
	\textbf{1.} By the biting lemma ({\it cf}. \cite{Ball-Murat}), there exist subsets $E_\epsi^i\subset M$ such that, for each $1\leq i\leq k=2$,
	\begin{itemize}
		\item $|E_\epsi^i|\rightarrow0$ as $\epsi\rightarrow0;$
		\item Sequence $|A_\epsi^i|^{p_i}\chi_{M\setminus E_\epsi^i}$ is equi-integrable.
	\end{itemize}
	Note that $|\cup_{j=1}^k E_\epsi^j|\rightarrow0$.
	
	Consider the modified sequence $X^i_\epsi := A^i_\epsi \chi_{M\setminus E_\epsi^i}$.
	Then $X^i_\epsi$ is $\L^{p_i}$-equi-integrable and converges weakly to $X^i\in\L^{p_i}$.
	
	\medskip
	\noindent
	\textbf{2.} We now claim that $\dd X^i_\epsi\Subset\W^{-1,p_i}$ for each $i=1,2$.
	Indeed, without loss of generality, it suffices to prove that,
	if $\dd X_\epsi^1\rightarrow0$ in $\W^{-1,1}$, then $\dd X_\epsi^1\rightarrow0$ in $\W^{-1,p_1}$.
	In view of the fact that $X^i_\epsi$ is uniformly bounded in $\L^{p_i}$,
	it is clear that $\sup_{\epsi>0}\Vert \dd X^1_\epsi\Vert_{\W^{-1,p_1}}<\infty$.
	It suffices to show that $\Vert \dd X^1_\epsi\Vert_{\W^{-1,p_1}}\rightarrow0$; that is,
	for any arbitrary compactly supported $\phi\in \Gamma(M,\mathfrak{g}\otimes\wedge^{n-d_1-1}T^*M)$
	with $\Vert \nabla\phi\Vert_{\L^{p_2}}\leq1$,
	$$
	\big|\int_M [X^1_\epsi\wedge \dd\phi]\big| \rightarrow 0 \qquad \mbox{uniformly in $\phi$ as $\epsi\to 0$}.
	$$
	
	To do this, following \cite{conti-div-curl}, for a fixed constant $\kappa>0$ (to be chosen),
	we consider a compactly supported $\kappa$-Lipschitz approximation $\psi$ of $\phi$ such that
	\begin{align*}
	& \|\dd\psi\|_{L^\infty}\le \kappa,\\
	& \big|\{x\in M\,:\,\phi(x)\neq\psi(x)\}\cup \{x\in M\,\:\,\nabla\phi(x)\neq\nabla\psi(x)\}\big|
	=:|S|\leq C(M)\kappa^{-p_2}.
	\end{align*}
	Then we have
	\begin{align}
	\left|\int_M [X^1_\epsi\wedge \dd\phi]\right|
	&\leq \left|\int_M [X^1_\epsi\wedge (\dd\phi-\dd\psi)]\right|+\left|\int_M [X^1_\epsi\wedge \dd\psi]\right|
	\label{5.4a}\\
	&\leq \left|\int_S [X^1_\epsi\wedge (\dd\phi-\dd\psi)]\right|+\Vert \dd\psi\Vert_{\L^\infty}\Vert \dd X_\epsi^1\Vert_{\W^{-1,1}}
	\nonumber\\
	&\leq\left( \int_S |X^1_\epsi|^{p_1}\right)^{\frac{1}{p_1}}\left(\int_S |\dd\phi-\dd\psi|^{p_2}\right)^{\frac{1}{p_2}}
	+\kappa \Vert \dd X_\epsi^i\Vert_{\W^{-1,1}}
	\nonumber\\
	&\leq \left( \int_S |X^1_\epsi|^{p_1}\right)^{\frac{1}{p_1}} \left(\Vert \dd\phi\Vert_{\L^{p_2}(S)}
	+ \kappa|S|^{\frac{1}{p_2}}\right)+\kappa \Vert \dd X_\epsi^i\Vert_{\W^{-1,1}}.
	\nonumber
	\end{align}
	Since $|X^1_\epsi|^{p_1}=|A_\epsi^1|^{p_1}\chi_{M\setminus E_\epsi^1}$ is equi-integrable,
	we may assume that there exists a positive function $\rho=\rho(s)$ defined for $s>0$, which is independent of $\epsi$,
	such that $\rho(s)\rightarrow0$ whenever $s\rightarrow0$ and
	\begin{equation}\label{5.5a}
	\int_S |X^1_\epsi|^{p_1}\leq \rho(|S|).
	\end{equation}
	Inserting \eqref{5.5a} into \eqref{5.4a}, we obtain
	\begin{align}
	\left|\int_M [X^1_\epsi\wedge \dd\phi]\right|
	&\leq \rho(C(M)\kappa^{-p_2})^{\frac{1}{p_1}}\left(\Vert \dd\phi\Vert_{\L^{p_2}(S)}
	+ \kappa|S|^{\frac{1}{p_2}}\right)+\kappa \Vert \dd X_\epsi^1\Vert_{\W^{-1,1}}\label{5.6a}\\
	&\leq \rho(C(M)\kappa^{-p_2})^{\frac{1}{p_1}}\left(1+(C(M))^{\frac{1}{p_2}}\right)+\kappa \Vert \dd X_\epsi^1\Vert_{\W^{-1,1}}.\nonumber
	\end{align}
	Note that $\rho(C(M)\kappa^{-p_2})\rightarrow0$ as $\kappa\to \infty$;
	and if, in addition, $\kappa$ grows slower than $\Vert \dd X^1_\epsi\Vert_{\W^{-1,1}}^{-1}$,
	the right-hand side of \eqref{5.6a} goes to zero. Then $\kappa$ may be taken as, for example,
	\begin{equation}
	\kappa =\kappa(\epsi)=\frac{1}{\sqrt{\Vert \dd X^1_\epsi\Vert_{\W^{-1,1}}}},
	\end{equation}
	to lead to the claim for $X^1_\epsi$. The proof for $X^2_\epsi$ is the same.
	
	\medskip
	\noindent
	\textbf{3.} Combining the results in Steps 1--2 above with Lemma \ref{lem-div-curl-lie-algebra},
	we conclude that
	$$
	[X^1_\epsi\wedge X^2_\epsi]\,\weakstar\,[X^1\wedge X^2] \qquad \mbox{in the sense of distributions.}
	$$
	
	Note that $X^i_\epsi\rightharpoonup X^i$ in $\L^{p_i}$.
	Since $|E^i_\epsi|\rightarrow0$, $X^i=A^i$ in $\L^{p_i}$.
	Similarly, $|E^1_\epsi\cup E^2_\epsi|\rightarrow 0$,
	so that $[X^1_\epsi\wedge X^2_\epsi] \weakstar [X^1\wedge X^2]=[A^1\wedge A^2]$.
	
	\medskip
	The general case can be obtained by following exactly
	the same argument as the case considered above and the same lines as in Lemma \ref{lem-div-curl-lie-algebra} and \cite{div-curl-rrt}.
	This completes the proof.
\end{proof}

\medskip	
\begin{proof}[\bf Proof of Theorem \ref{thm-weak-cty-equi}]
	\noindent
	Let $u_\epsi$ be a sequence of weak isometric immersions of $(M,g)$ in $\R^{N}=\R^{n+1}$ for $p\geq2$.
	We recall that, in the codimension one case, the normal bundle $NM$ is a line bundle.
	Hence, we may view $\nu_\epsi:NM\rightarrow T\R^{n+1}|_{u(M)}$ simply as the Gau\ss{} map.
	We divide the proof into four steps.		
	
	\medskip
	\noindent
	{\bf 1.} Let $E\subset M$ be any bounded domain, and let $\alpha$ be an orthonormal coframe on $TM|_E$ with respect to $g$.
	As in the proof of Theorem \ref{thm-weak-cty}, we now have the following sequence of the Darboux coframes for $\epsi>0$:
	\begin{equation}
	\theta_\epsi = \tau_\epsi^*\alpha+\eta_\epsi,
	\end{equation}
	where $\eta_\epsi$ is the co-vector dual to $\nu_\epsi$ for each fixed $\epsi>0$.
	Similarly, let $\theta=\tau^*\alpha+\eta$.
	It is clear that $\theta_\epsi\rightharpoonup\theta$ weakly in $\W^{1,p}(E)$ and weak-star in $\L^\infty(E)$.
	
	\smallskip
	\noindent
	{\bf 2.} Let $\omega_\epsi$ be the connection form of $\theta_\epsi$ for $\epsi>0$.
	As usual, we decompose the connection form into its tangential, normal, and $\II$ parts;
	however, in codimension one, the normal connection is always zero, so that
	\begin{equation*}
	\omega_\epsi=\begin{pmatrix}
	\omega_\epsi^\top & -\omega_\epsi^\II\\[1mm]
	\t\omega_\epsi^\II &0
	\end{pmatrix}\in\so(n+1).
	\end{equation*}

	\smallskip
	\noindent
	{\bf 3.} We now claim that $\omega_\epsi^\top = \tau_\epsi^*\beta$ for $\epsi>0$,
	where $\beta$ is the connection form associated to the tangent coframe $\alpha$,
	{\it i.e.}, $\dd\alpha=\beta\wedge\alpha$.
	
	Indeed, we have
	\begin{align*}
	\dd\theta^\top_\epsi|_{TM}= \dd(\tau^*_\epsi\alpha)
	= \tau^*_\epsi \dd\alpha
	= \tau^*_\epsi (\beta\wedge\alpha)
	= \tau^*_\epsi \beta \wedge\theta^\top_\epsi|_{TM}.
	\end{align*}
	By uniqueness, this leads to the claim.
	
	As a consequence, we conclude that
	$\omega_\epsi^\top$
	converges to
	$\omega^\top$ in $\W^{1,p}(E)$ and strongly in $\L^2(E)$.
	In addition, we have
	\begin{align*}
	\dd\omega_\epsi^\top +\omega^\top_\epsi\wedge\omega^\top_\epsi
	= \tau_\epsi^*(\dd\beta+\beta\wedge\beta)
	= \tau_\epsi^*\Omega
	= \tau_\epsi^*\Rm(\alpha\wedge\alpha),
	\end{align*}
	where $\Rm$ is the Riemann curvature endomorphism associated to $g$.
	
	Furthermore, recall that $\theta_\epsi$ is weakly convergent in $\W^{1,2}(E)$.
	Hence, the connection form $\omega_\epsi$ is uniformly bounded in $\L^2(E)$,
	so that $\omega^\II_\epsi$ is uniformly bounded in $\L^2(E)$.
	This implies that the week limit still obeys the Codazzi equation.

	\smallskip
	\noindent
	{\bf 4.} We now prove that $\t\omega^\II_\epsi\wedge\omega^\II_\epsi$ is equi-integrable.
	Fix $\epsi>0$ and an arbitrary $F\subset E$.
	By the Gau\ss{}  equation, we have
	\begin{equation}
	\int_F|\t\omega_\epsi^\II\wedge\omega_\epsi^\II| = \int_F |\dd\omega^\top_\epsi+\omega^\top_\epsi\wedge\omega^\top_\epsi|
	= \int_F|\mathbf{Rm}(\alpha\wedge\alpha)|,
	\end{equation}
	which is an $\L^1$ function that is independent of $\epsi$, so that
	\begin{equation}
	\limsup_{F\subset E, |F|\rightarrow0}\sup_{\epsi>0}\int_F |\t\omega^\II_\epsi\wedge\omega^\II_\epsi|=0;
	\end{equation}
	that is, $\t\omega^\II_\epsi\wedge\omega^\II_\epsi$ is equi-integrable.
	
	The conclusion now follows from Lemma \ref{lem-div-curl-equi}, as in the proof of Theorem \ref{thm-weak-cty}.
\end{proof}

\begin{rk}
	The calculation for showing the equi-integrability of $\t\omega^\II\wedge\omega^\II$ holds for arbitrary codimension $N-n\geq1$,
	and hence the Gau\ss{} equation is always weakly continuous along a sequence of isometric immersions converging weakly in $\W^{2,2}_\loc$.
	However, if $N>n+1$, the Ricci equation contains a further quadratic term $\omega^\perp\wedge\omega^\perp$,
	for which the argument of the proof is unavailable, as the "normal curvature" is not controlled {\it a priori}.
\end{rk}

\subsection{$L^p$--energy minimizers of isometric embeddings}\label{minimizers}\label{subsection-minimisers}
As a simple application of Theorems \ref{thm-weak-cty} and \ref{thm-weak-cty-equi}, we provide a "selection criterion" for isometric immersions
based on the second fundamental form $\II$. 
Indeed, a general philosophy \cite{gromov,gromov-columbus} on isometric embeddings is that,
if a Riemannian manifold embeds in a larger manifold with sufficiently high codimension, then there should be an abundance of isometric immersions,
so that it is desirable to have analytical means of selecting a "best representative" of the isometric immersions of a manifold.
Only in the lower codimension do we expect to see some rigidity phenomena: see \eg \cite{bbg,cdls,pogorelov} and the references therein. Such a strategy is also commonplace in elasticity: see \eg \cite{ciarlet-2,ciarlet-3} and the references therein.

Our next result shows that the $\L^p$ norm of the second fundamental form $\II$ is such a criterion.

\begin{thm}
	Let $(M,g)$ be a closed Riemannian manifold, let the set of isometric immersions of $(M,g)$ into $\R^N$ be non-empty for some $N\in\N$, and let $p>2$.
	Then there exists a weak $\W^{2,p}$ isometric immersion $u\,:\,M\rightarrow \R^N$ minimising the $\L^p$ norm
	of the second fundamental form $\II$ among all isometric immersions.
	
	If $N=n+1$, the same result holds for $p\geq2$.
\end{thm}

The proof follows at once from considering the functional
\begin{equation}
\mathcal{E}(u)=\int_M |\II|^p,
\end{equation}
defined for any immersion $\,u:\,M\rightarrow\R^N$.
This functional is clearly bounded from below.
Moreover, restricted to the set of isometric immersions, it is weakly continuous.
We then consider a minimising sequence $u_\epsi$, whose limit $u$ is a critical point of $\mathcal{E}$.
Theorem \ref{thm-weak-cty} indicates that, as $u_\epsi\rightharpoonup u$ in $\W^{2,p}$,
the weak limit $u$ is still an isometric immersion satisfying the Gau\ss--Codazzi--Ricci equations.

Theorem \ref{thm-weak-cty-equi} yields the last assertion when $N=n+1$.

\subsection{Semi-Riemannian manifolds}\label{subsection-semi}
Theorem \ref{thm-weak-cty} may be extended straightforwardly to the case of semi-Riemannian manifolds: in fact, the Riemannian character of either the base or the target manifold plays no role in the proof of Theorem \ref{thm-weak-cty}.
The difference lies in the structure groups of the underlying principal bundles: in the case of a semi-Riemannian manifold,
the Darboux bundle is a principal bundle, the structure group of which is the indefinite special orthogonal group $\SO(n,k)$.

More precisely, we recall that a quadratic form $g$ on a vector space has signature $(n,k)$ if the maximal dimension of the subspace
on which $g$ is positive definite is $n$, and the maximal dimension of the subspace on which $g$ is negative definite is $k$.
For a fixed quadratic form $g$ of signature $(n,k)$, we denote by $\O(n,k)$ the group of linear transformations leaving $g$ invariant,
and by $\SO(n,k)$ the subgroup of $\O(n,k)$ with determinant one. Its Lie algebra $\so(n,k)$ is the set of matrices $X\in M_{n+k}(\R)$ satisfying $\t XI_{n,k}=-I_{n,k}X$, where we have denoted 
$$I_{n,k} =\begin{pmatrix}
I_n & 0\\ 0 & I_k,
\end{pmatrix}$$
and $I_n$ and $I_k$ are identity matrices of dimension $n$ and $k$ respectively.

Just as in the Riemannian case, the Gau\ss--Codazzi--Ricci equations for immersed semi-Riemannian manifolds
are a consequence of the Cartan equations.
Let $(M,g)$ be a semi-Riemannian manifold with signature $(n,k)$ and  $(\widetilde{M},\tilde{g})$ another semi-Riemannian manifold with signature $(n+r,k+s)$.
A necessary condition for the existence of an immersion $u\,:\,(M,g)\rightarrow (\widetilde{M},\tilde{g})$ is that $r$ and $s$ be non-negative integers.
Then the connection form decomposes, as in \S \ref{section-gcr}, into a tangential part $\omega^\top\in\so(n,k)$,
a normal part $\omega^\perp\in\so(r,s)$, and a second fundamental form $\omega^\II$ so that
\begin{equation}
\begin{pmatrix}
\omega^\top &\t\omega^\II\\
-\omega^\II &\omega^\perp
\end{pmatrix}\in\so(n+k,r+s).
\end{equation}

Following the proof of Theorem \ref{thm-weak-cty}, we have the following theorem:

\begin{thm}\label{thm-weak-cty-semi}
	Let $(M,g)$ be a semi-Riemannian manifold with signature $(n,k)$,
	and let $u_\epsi$ be a sequence of weak $\W^{2,p}_\loc$ isometric immersions into  another semi-Riemannian manifold $(\widetilde{M},\tilde{g})$
	with signature $(n+r,k+s)$, and
	let $\II_\epsi\in\L^p_\loc(M,\Hom(TM\times TM,NM))$ be the second fundamental forms associated to each $u_\epsi$ and satisfy
	\begin{equation}
	\sup_{\epsi>0}\Vert\II_\epsi\Vert_{\L^p(E)}<\infty \qquad\mbox{for $p>2$ and for any bounded domain $E\subset M$}.
	\end{equation}
	Then, up to a subsequence, $u_\epsi\rightharpoonup u$ in $\W^{2,p}_{\loc}$ such that $u$ is still a weak $\W^{2,p}_{\loc}$ isometric immersion
	and any orthonormal coframe on $(M,g)$ may be extended into a Darboux coframe in $(\widetilde{M},\tilde{g})$,
	adapted to $u$ and satisfying the Gau\ss--Codazzi--Ricci equations \eqref{system-gcr} in the sense of distributions.
\end{thm}

\begin{rk}
	The weak continuity of the Cartan equation has also been studied specifically in {\rm \cite{chen-li-semi}}. The compensated compactness theorem on semi-Riemannian manifolds proved in \cite{chen-li-semi} may be used to give an alternative proof of Lemma {\rm \ref{lem-div-curl-lie-algebra}},
	and thus Proposition {\rm \ref{thm-main-chapter-connections}}.
\end{rk}

\medskip
\appendix

\section{Convergence of the Metric Tensors}\label{section-very-weak}
We prove a very mild sufficient condition for the weak continuity of the metric tensor along a sequence of isometric immersions.

\begin{thm}\label{thm-very-weak-II}
	Let $(M,g)$ be a Riemannian manifold.
	Consider a sequence of weak isometric immersions $u_\epsi$ such that, for any bounded domain $E\subset M$,
	\begin{equation}\label{eqn-hypothesis-very-weak}
	\sup_{\epsi>0}\Vert H_\epsi\Vert_{\mathcal{M}(E)}<\infty,
	\end{equation}
	where $H_\epsi=\tr(g^{-1}\II_\epsi)$ are the vector-valued mean curvatures,
	and
	$$
	u_\epsi\weakstar u  \qquad \mbox{in $W^{1,\infty}_{\loc}$ $($up to a non-relabelled subsequence$)$}.
	$$
	Then $u$ is an isometric immersion.
\end{thm}

\begin{proof} We divide the proof into two steps.

	\medskip		
	\noindent
	{\bf 1}. Let $u_\epsi\,:\,(M,g)\rightarrow\R^N$ be a sequence of isometric immersions
	with locally uniformly bounded second fundamental forms.
	Then $\laplace_gu_\epsi$ has locally uniformly bounded total mass. Indeed, it is easy to verify that, in any local coordinate system on $M$,
	a weak $\W^{2,p}$ isometric immersion $u_\epsi$ satisfies
	\begin{equation*}
	\partial^2_{ij}u_\epsi = \Gamma^k_{ij}\partial_k u_\epsi + \II^m_{ij,\epsi}\nu_{m,\epsi},
	\end{equation*}
	where $\nu_{m,\epsi}$ are $\R^N$-valued orthonormal vectors that are normal to $u_\epsi(M)$ and the indices range over $1\leq i,j\leq n$ and $1\leq m\leq N-n$. Taking the trace, we see that
	\begin{equation}\label{eqn-harmonic-map}
	\laplace_g u_\epsi =	H^\epsi=g^{ij}\II^m_{ij,\epsi}\nu_{m,\epsi}.
	\end{equation}
	We conclude directly from \eqref{eqn-hypothesis-very-weak} that
	$\laplace_gu_\epsi$ has a uniformly bounded total mass locally.
	
	\smallskip
	\noindent
	{\bf 2.}
	Let $E\subset M$ be a bounded
	domain.
	From Step 1, $\{u_\epsi\}$ is a uniformly bounded sequence in $\W^{1,\infty}(E)$ such that
	\begin{equation*}
	\sup_{\epsi>0}\Vert \laplace_g u_\epsi\Vert (E)<\infty.
	\end{equation*}
	Then
	$$
	\{\laplace_g u_\epsi \}\qquad \mbox{is compact in $W^{-1,q}_\loc$ for $q\in (1, \frac{n}{n-1})$},
	$$
	which implies that
	$$
	\{u_\epsi\} \qquad \mbox{is compact in $W^{1,q}_\loc$ for $q\in (1, \frac{n}{n-1})$}.
	$$
	Combining the uniform boundedness of $u_\epsi$ in $\W^{1,\infty}(E)$ with interpolation, we
	conclude that
	$$
	\{u_\epsi\} \qquad \mbox{is pre-compact in $\W^{1,2}(E)$.}
	$$
	By Sobolev embedding, this is sufficient to pass to the limit almost everywhere in the isometry constraint.
\end{proof}

\begin{rk}\label{rk-nash-kuiper}
	The isometry constraint is not continuous with respect to the weak* topology of $\W^{1,\infty}$, and the result of Theorem \ref{thm-very-weak-II} is false if hypothesis \eqref{eqn-hypothesis-very-weak} is removed.
	Indeed, let $u\,:\,M\rightarrow\R^N$ be a smooth short immersion of $(M,g)$ in $(\R^N, \euc)$, i.e.,
	$\langle \dd u, \dd u\rangle<g$ in the sense of bilinear forms.
	As $M$ is compact, this is always possible, provided that the codimension $N-n$ is large enough.
	By the Nash--Kuiper theorem {\rm \cite{nash-1,kuiper}}, $u$ may be uniformly approximated by a sequence of
	immersions $u_\epsi\in C^{1}(M;\R^N)$ such that
	\begin{align*}
	u_\epsi\xrightarrow{C^0} u,\qquad u_\epsi^*\euc=\langle \dd u_\epsi,\dd u_\epsi\rangle = g.
	\end{align*}
	Taking the sup-norms on both sides of the second equation above, we have
	$$
	\Vert \dd u_\epsi\Vert_{\L^\infty}\leq \Vert g\Vert_{\L^\infty}^{\frac{1}{2}}.
	$$
	Therefore, up to a subsequence, $u_\epsi\weakstar u$ in the weak-star topology of $\W^{1,\infty}$.
	This shows that the isometry constraint is not continuous with respect to the weak* topology of $\W^{1,\infty}$.
\end{rk}

\section{Div-Curl Lemmas for Elliptic Complexes}\label{appendix-div-curl}
In this appendix, for self-containedness and clarity of the exposition of the paper,
we present some facts about elliptic complexes on manifolds
and present a div-curl lemma for this case.
Without loss of generality, we assume that $M$ is a {\it closed} Riemannian manifold in this appendix; otherwise, we may
restrict our presentation to a bounded domain in $M$.

Let $\mathfrak{E}:=\{E(i)\}_{i\in\N}$ be a sequence of fiber bundles on $M$.
A complex is a sequence of linear maps $\mathfrak{D}:=\{\DD (i)\,:\, \Gamma(M,E(i))\rightarrow\Gamma(M,E(i+1))\}_{i\in\N}$.
The complex $\mathfrak{D}$ is differential if $\DD (i+1)\circ \DD (i)=0$, or equivalently if the sequence:
\begin{equation*}
\cdots\xrightarrow{\DD(i-1)}\Gamma(M,E(i))\xrightarrow{\DD(i)}\Gamma(M,E(i+1))\xrightarrow{\DD(i+1)}\Gamma(M,E(i+2))\xrightarrow{\DD(i+2)}\cdots
\end{equation*}
is exact.

We are interested in the case that $\mathfrak{E}$
is a complex of vector bundles and $\mathfrak{D}$ is a sequence of linear differential operators.
We assume that $\mathfrak{E}$
is also equipped with an inner product.
Let $*$ denote the associated Hodge duality operator, and $(*1)$ the volume form.
For an element $\phi\in\Gamma(M,E(i))$ and $1\leq p<\infty$, we can define the $\L^p$ norms
\begin{equation*}
\Vert \phi\Vert_{\L^p}^p=\int_M |\phi|^p(*1).
\end{equation*}
The Lebesgue spaces on $\mathfrak{E}$
are then the completion of $\Gamma(M,E(i))$ with respect to the above norm.
We also assume that $\Gamma(M,E(i))$ is equipped with a sequence of fixed
connections $\nabla(i)$,
so that the Sobolev spaces $\W^{k,p}(M,E(i))$ may be defined by acting with $\nabla(i)$ as
\begin{equation*}
\Vert \phi\Vert^p_{\W^{1,p}} =\int_M |\phi|^p(*1)+ \int_M |\nabla(i)\phi|^p(*1),
\end{equation*}
and the spaces $\W^{-k,p}$ are defined by duality in the standard way.

Fix local coordinates on $M$.  In the associated local trivializations,
$$
\{\partial_j\,:\,  1\leq j\leq \dim(E(i))\}
$$
of the bundles $E(i)\rightarrow M$, the operators $\DD(i)$ take the form:
\begin{equation*}
\DD(i) = \sum_{|\mu|\leq k(i)} P^\mu\partial_\mu,
\end{equation*}
where $\mu\in\N^{\dim(E(i))}$ is a multi-index, $\partial_\mu:=\otimes_{1\leq j \leq \dim(E(i))} \partial^{\mu_j}$, and
$$
P^\mu=P^\mu(i)\in{\Gamma(M,\Hom(E(i),E(i+1)))}=\Gamma(M,\mathscr{L}(E(i), E(i+1))).
$$
The operator $\DD(i)$ is of order $k(i)$.

Recall that the symbol of $\DD(i)$,
denoted by $\sigma_{\DD(i)}$, is a vector bundle complex
\begin{equation*}
\cdots\xrightarrow{\sigma_{\DD(i-1)}}\pi^*(E(i))\xrightarrow{\sigma_{\DD(i)}}\pi^*(E(i+1))\xrightarrow{\sigma_{\DD(i+1)}}\cdots
\end{equation*}
where we write $T^*M\xrightarrow{\pi} M$ for the cotangent bundle,
$\pi^*E$ for the vector bundle $\pi^*E(i)\rightarrow T^*M$ pulled back over $T^*M$ by the fibration map $\pi$
(see Figure \ref{figure-pullback-cotangent-bundle}).
Thus, for a fixed $\xi\in T^*_xM$, ${\sigma_{\DD(i)}(\xi)\,:\,E(i)_x\rightarrow E(i+1)_x}$ is a linear map,
given as the symbol of the differential operator $\DD(i)$ in the coordinates on $E(i)$.
Alternatively, it may be defined directly as follows:
Let $f\in\Gamma(M,E(i))$ be such that $f(x)=v$, and let $g\in C^\infty(M)$ be such that $g(x)=0$ and $\dd g|_x = \xi$.
Then
$$
\sigma_{\DD(i)}(\xi)v=\DD(i)(\frac{g^{k(i)}f}{k(i)!})(x)\in E(i+1)
$$
so that
$$
\sigma_{\DD(i)}(\xi)\in\mathscr{L}(E(i),E(i+1)).
$$
The complex, $\mathfrak{D}=\{\DD(i)\}_{i\in\N}$, is elliptic if $\sigma_{\DD(i)}$ is exact for all $i\in\N$.

\begin{figure}[h!]
	\centering
	\begin{tikzcd}
	\pi^*E(i) \arrow[r] \arrow[d]
	& E(i) \arrow[d] \\
	T^*M \arrow[r, "\pi"]
	& M
	\end{tikzcd}
	\caption{Pull-back of the bundles $E(i)\rightarrow M$ on the cotangent bundle by the fibration $\pi:T^*M\rightarrow M$. This diagram is commutative.}\label{figure-pullback-cotangent-bundle}
\end{figure}

The $\L^2$ adjoint of $\DD(i)$ is denoted by $\DD^*(i)$ and forms a new complex
\begin{equation*}
\cdots\xrightarrow{\DD^*(i+1)}\Gamma(M,E(i+1))\xrightarrow{\DD^*(i)}\Gamma(M,E(i))\xrightarrow{\DD^*(i-1)}\cdots
\end{equation*}
which is also exact if complex $\mathfrak{\DD}$
is differential; similarly for its symbol $\{\sigma_{\DD^*(i)}\}_{i\in\N}$.
The Hodge Laplacian $\laplace(i)$ of complex $\{\DD(i)\}_{i\in\N}$ is the sequence of operators given by
\begin{equation*}
\laplace(i)=\DD(i-1)\DD^*(i-1) + \DD^*(i)\DD(i)\,:\,\Gamma(M,E(i))\rightarrow\Gamma(M,E(i)).
\end{equation*}
It is well-known that, if $\mathfrak{D}$ is an elliptic complex,
the Hodge Laplacian $\laplace(i)$ is elliptic as $\sigma_{\laplace(i)}$ is an isomorphism.

The following commutation property is standard:

\begin{lem}\label{lem-hodge-commute}
	Let $M$ be a smooth closed Riemannian manifold,
	and let $(\mathfrak{E},\mathfrak{D})$ be an elliptic complex on $M$, $\mathfrak{D}^*$ its adjoint complex,
	and $\{\laplace(i)\}_{i\in \N}$ the associated Hodge Laplacian.
	Then
	$$
	\laplace(i) \DD(i)=\DD(i)\laplace(i).
	$$
	Moreover, for  the inverse Laplacian $\laplace^{-1}(i)$, we have
	$$
	\laplace^{-1}(i)\DD(i)=\DD(i)\laplace^{-1}(i).
	$$
\end{lem}

\medskip
For our purpose, it suffices to consider the case that all the operators $\DD(i)$ are first-order operators.
We note that it is possible to derive the entire theory expounded in this section for general elliptic complexes of arbitrary order $k\in\N$,
in which the Hodge Laplacian is an operator of order $2k$ in general.

By the standard elliptic theory, the Hodge Laplacian $\laplace(i)$ of a first-order elliptic complex admits a continuous inverse $\laplace^{-1}(i):\W^{-1,p}(M,E(i))\rightarrow\W^{1,p}(M,E(i))$ for $1<p<\infty$.
The next observation is essentially due to \cite{div-curl-rrt}.

\begin{lem}\label{lem-hodge-decomp}
	Consider a sequence of sections $\alpha_\epsi$ so that $\alpha_\epsi \rightarrow\alpha$ weakly in $L^p$
	and $\DD(i)\alpha_\epsi$ is compactly contained in $\W^{-1,p}$.
	Then there exist $\rho_\epsi\in\Gamma(M,E(i))$ and $\psi_\epsi\in\Gamma(M,E(i-1))$ satisfying that
	$\rho_\epsi\rightarrow\rho$ and $\psi_\epsi\rightarrow\psi$ strongly in $\L^p$ and
	$\DD(i)\psi_\epsi\rightharpoonup \DD(i) \psi$ weakly in $\L^p$ such that $\alpha_\epsi = \DD(i)\psi_\epsi+\rho_\epsi$
	and $\alpha=\DD(i)\psi+\rho$.
\end{lem}
\begin{proof}
	We first note the following algebraic identity:
	\begin{align*}
	\alpha_\epsi = \laplace(i)\laplace^{-1}(i)\alpha_\epsi=  \DD(i)\DD^*(i)\laplace^{-1}(i)\alpha_\epsi+\DD^*(i)\DD(i)\laplace^{-1}(i)\alpha_{\epsi}.
	\end{align*}
	Denoting $\psi_\epsi =\DD^*(i)\laplace^{-1}(i)\alpha_\epsi$ and $\rho_\epsi = D^*(i)D(i)\laplace^{-1}(i)\alpha_\epsi$,
	we see that $\alpha_\epsi = \DD(i)\psi_\epsi + \rho_\epsi$.
	
	By Lemma \ref{lem-hodge-commute}, $\rho_\epsi=\DD^*(i)\DD(i)\laplace^{-1}(i)\alpha_\epsi=\DD^*(i)\laplace^{-1}(i)\DD(i)\alpha^{\epsi}$.
	Since $\DD(i)\alpha_\epsi$ is compactly contained in $\W^{-1,p}$ and $\laplace^{-1}(i)$ is continuous,
	$\rho_\epsi$ is compactly contained in $\L^p(M,E(i))$, which implies that $\rho_\epsi$ converges strongly to a limit $\rho$ in $L^p$.
	
	On the other hand, $\psi_\epsi$ is in a bounded set of $\W^{1,p}(M,E(i-1))$ so that $\psi_\epsi$ converges to a limit $\psi$ weakly
	in $\W^{1,p}(M,E(i-1))$ and strongly in $\L^p$, by the Sobolev compact embedding theorem.
	Finally, by uniqueness of weak limits, we have
	$$
	\alpha = \DD(i)\psi + \rho.
	$$
\end{proof}

This observation allows us to obtain the following general div-curl lemma for elliptic complexes.

\begin{lem}\label{lem-div-curl-step}
	Let $p,q>1$ be such that $\frac{1}{p}+\frac{1}{q}=1$, and let $\alpha_\epsi\in\L^p(M,E(i))$,
	and $\beta_\epsi\in\L^q(M,E(i))$ be bounded sequences such that
	\begin{align}
	\DD(i)\alpha_\epsi\Subset\W^{-1,p}(M,E(i+1)),\qquad
	\DD^*(i)\beta_\epsi\Subset\W^{-1,q}(M,E(i-1)).
	\end{align}
	Then, for all $\phi\in C^\infty_c(M)$,
	\begin{equation*}
	\int_M \langle \alpha_\epsi,\beta_\epsi\rangle\phi\rightarrow\int_M \langle \alpha,\beta\rangle\phi,
	\end{equation*}
	which is denoted as $\langle \alpha_\epsi,\beta_\epsi\rangle\rightharpoonup \langle\alpha,\beta\rangle$
	when no ambiguity arises.
\end{lem}

\begin{proof}
	We view  sequence $\alpha_\epsi\in\L^p(M,E(i))$ as a part of complex $(\mathfrak{E},\mathfrak{D})$, and $\beta_\epsi\in \Gamma(M,E(i))$
	as a part of complex $(\mathfrak{E},\mathfrak{D}^*)$.
	Applying Lemma \ref{lem-hodge-decomp} yields
	$$
	\alpha_\epsi = \DD(i)\psi_\epsi+\rho_\epsi, \qquad \beta_\epsi = \DD^*(i)\zeta_\epsi+\xi_\epsi.
	$$
	Take $\phi\in C^\infty_c(M)$. Then we have
	\begin{align*}
	\int_M \langle \alpha_\epsi,\beta_\epsi\rangle\phi
	=& \int_M \langle \DD(i)\psi_\epsi,\DD^*(i)\zeta_\epsi\rangle \phi
	+\int_M \langle \DD(i)\psi_\epsi,\xi_\epsi\rangle\phi\\
	&+\int_M \langle \rho_\epsi,\xi_\epsi\rangle\phi
	+ \int_M \langle \rho_\epsi, \DD^*(i)\zeta_\epsi\rangle\phi.
	\end{align*}
	
	By the H\"{o}lder inequality, $\langle \rho_\epsi,\xi_\epsi\rangle \rightarrow\langle \rho,\xi\rangle$ strongly in $\L^1$,
	so that $\int_M \langle \rho_\epsi,\xi_\epsi\rangle\phi\rightarrow\int_M \langle \rho,\xi\rangle\phi$.
	
	The second terms and fourth terms are treated similarly. For example, $\DD(i)\psi_\epsi$ converges weakly in $\L^p$,
	while $\xi_\epsi$ converges strongly in $\L^q$, which leads to
	$$
	\int_M \langle \DD(i)\psi_\epsi,\xi_\epsi\rangle\phi\rightarrow\int_M \langle \DD(i)\psi,\xi\rangle\phi.
	$$
	
	Finally, we focus on the first term. Recall that $\DD(i)$ is a first-order linear differential operator
	so that, for all $i\in\N$, we may write
	\begin{equation*}
	\DD(i)=\sum_{|\mu|\leq1}P_\mu(i) \partial^\mu.
	\end{equation*}
	By the Leibniz rule, we have
	\begin{align*}
	\DD(i+1)(\phi \DD(i)\psi_\epsi) &= \sum_{|\mu|\leq1}
	\sum_{\kappa\leq\mu}  {\mu\choose \kappa} \partial^\kappa\phi P_\mu(i+1)\partial^{\mu-\kappa} \DD(i)\psi_\epsi\\
	&= \phi \DD(i+1)\DD(i)\psi_\epsi + \sum_{|\mu|\leq1}
	\sum_{0\neq \kappa\leq\mu} \partial^\kappa\phi P_\mu(i+1)\partial^{\mu-\kappa}\DD(i)\psi_\epsi.
	\end{align*}
	Notice that $\mathfrak{D}$ is an exact complex so that $\DD(i+1)\DD(i)=0$.
	Thus, we are left with
	\begin{equation}\label{A.11}
	\DD(i+1)(\phi \DD(i)\psi_\epsi)= \sum_{|\mu|\leq1}
	\sum_{0\neq \kappa\leq\mu}\partial^\kappa\phi P_\mu(i+1)\partial^{\mu-\kappa}D(i)\psi_\epsi.
	\end{equation}
	We note that $\DD(i)$ is a first-order operator
	so that, whenever $0\neq\kappa\leq\mu$ and $|\mu|\leq1$, we conclude that $\mu=\kappa$.
	Then the right-hand side of \eqref{A.11} simplifies further to
	\begin{equation*}
	\DD(i+1)(\phi \DD(i)\psi_\epsi)= \sum_{|\mu|\leq1} \partial^\mu\phi P_\mu(i+1) \DD(i)\psi_\epsi.
	\end{equation*}
	
	We now integrate by parts to obtain
	\begin{align*}
	\int_M \langle \phi \DD(i)\psi_\epsi, \DD^*(i)\zeta_\epsi\rangle
	&= -\int_M \langle \DD(i)(\phi \DD(i)\psi_\epsi),\zeta_\epsi\rangle\\
	&= - \sum_{|\mu|\leq1}
	\int_M  \partial^\mu\phi \left\langle P_\mu(i+1) \DD(i)\psi_\epsi,\zeta_\epsi\right\rangle.
	\end{align*}
	Recall that, by construction, $\zeta_\epsi$ converges to $\zeta$ strongly in $\L^q$.
	Furthermore, $\DD(i)\psi_\epsi$ converges weakly in $\L^p$.
	Therefore, we have
	\begin{align*}
	\int_M \langle \phi \DD(i)\psi_\epsi, \DD^*(i)\zeta_\epsi\rangle
	&= - \sum_{|\mu|\leq1}\int_M \partial^\mu\phi \left\langle P_\mu(i+1)\DD(i)\psi_\epsi,\zeta_\epsi\right\rangle\\
	&\xrightarrow{\epsi\rightarrow0}- \sum_{|\mu|\leq1}
	\int_M  \partial^\mu\phi \left\langle P_\mu(i+1) \DD(i)\psi,\zeta\right\rangle\\
	&\quad\quad\,\, = \int_M \langle\phi\DD(i)\psi, \DD^*\zeta\rangle,
	\end{align*}
	where we have integrated by parts once more for the last equality.
	This completes the proof.
\end{proof}

Lemma \ref{lem-div-curl-step} is a generalization of the standard div-curl lemma.
Indeed, as an example of the approach in this appendix, we consider the complex:
\begin{equation*}
0\rightarrow \Gamma(M)\xrightarrow{\grad} \Gamma(M,TM)\xrightarrow{\curl}\Gamma(M,TM\otimes TM)\rightarrow\cdots
\end{equation*}
whose adjoint complex is
\begin{equation*}
\cdots\rightarrow\Gamma(M,TM)\xrightarrow{\div}\Gamma(M)\rightarrow0.
\end{equation*}
For instance, if $\dim M=3$, we have the short exact sequence:
\begin{equation*}
0\rightarrow \Gamma(M)\xrightarrow{\grad} \Gamma(M,TM)\xrightarrow{\curl} \Gamma(M,TM)\xrightarrow{\div} \Gamma(M)\rightarrow0,\quad\dim M=3.
\end{equation*}
Applying Lemma \ref{lem-div-curl-step} to these complexes,
we obtain the classical div-curl lemma \cite{tartar-pde,murat-compcomp1}: {\it Let $u_\epsi$ and $v_\epsi$ be sequences of vector fields bounded in $\L^p$ and $\L^q$, respectively,
	and $1<p,q<\infty$ such that $\div u_\epsi\Subset \W^{-1,p}$ and $\curl v_\epsi\Subset\W^{-1,q}$.
	Then $\langle u_\epsi,v_\epsi\rangle\weakstar \langle u,v\rangle$}.

We now focus on the de Rham complex:
\begin{equation}\label{eqn-de-rham-complex}
\cdots\xrightarrow{\dd}\Gamma(M,\wedge^kT^*M)\xrightarrow{\dd}\Gamma(M,\wedge^{k+1}T^*M)\xrightarrow{\dd}\cdots
\end{equation}
where $\wedge^kT^*M$ is the alternate $k$-th power of the cotangent bundle ({\it i.e.}, $k$-forms on $M$),
and $\dd$ is the exterior derivative.
We denote by $\delta$ the adjoint of $\dd$.
Since $M$ is a Riemannian manifold, $\wedge^kT^*M$ is equipped with an inner product.
We obtain a div-curl lemma for the products of differential forms.
As before, let $\alpha_\epsi\in\L^p(M,\wedge^kT^*M)$ and $\beta_\epsi\in\L^q(M,\wedge^kT^*M)$
such that $\dd\alpha_\epsi\Subset\W^{-1,p}(M,\wedge^{k+1}T^*M)$ and  $\delta\beta_\epsi\Subset\W^{-1,q}(M,\wedge^{k-1}T^*M)$.
Then, by Lemma \ref{lem-div-curl-step}, 
\begin{equation*}
\int_M \langle \alpha_\epsi,\beta_\epsi\rangle\phi\rightarrow\int_M \langle \alpha,\beta\rangle\phi
\qquad\,\,\mbox{for all $\phi\in C^\infty_c(M)$}.
\end{equation*}

Note that
$$
\int_M \langle \alpha,\beta\rangle\phi=\int_M\phi(\alpha\wedge*\beta),
$$
and the confinement hypothesis for $\delta\beta_\epsi$ is equivalent to the statement that
$$
\dd(*\beta_\epsi)\Subset\W^{-1,q}(M,\wedge^{n-k+1}T^*M).
$$
Thus, using the graded structure of the de Rham complex, we see that the following result holds:
{\it Let $\alpha_\epsi\in\L^p(M,\wedge^kT^*M)$ and $\beta_\epsi\in\L^q(M,\wedge^mT^*M)$
	such that
	$$
	\dd\alpha_\epsi\Subset\W^{-1,p}(M,\wedge^{k+1}T^*M), \qquad
	\dd\beta_\epsi\Subset\W^{-1,q}(M,\wedge^{m+1}T^*M).
	$$
	Then, for any smooth, compactly supported $\phi\in \Gamma(M, \wedge^{n-k-m}T^*M)$,
	\begin{equation*}
	\int_M \phi\wedge\alpha_\epsi\wedge\beta_\epsi\rightarrow\int_M \phi\wedge \alpha\wedge\beta.
	\end{equation*}
}
This corresponds to a special case of the main result of \cite{div-curl-rrt}.
By induction, the corresponding general result of \cite{div-curl-rrt} can be obtained,
which concerns general alternating products of differential forms.

\end{document}